\documentclass[11pt]{amsart}
\allowdisplaybreaks

\usepackage{amsmath,amssymb,amsthm,tikz, nicefrac, mathrsfs,upgreek, soul}
\usepackage{mathtools}
\usepackage{mathrsfs}
\usepackage{enumitem}
\usepackage[normalem]{ulem}

\RequirePackage{color}
\RequirePackage[colorlinks, urlcolor=my-blue,linkcolor=my-blue,citecolor=my-blue]{hyperref}
\definecolor{my-blue}{rgb}{0.0,0.0,0.6}
\definecolor{my-red}{rgb}{0.5,0.0,0.0}
\definecolor{my-green}{rgb}{0.0,0.5,0.0}
\definecolor{nicos-red}{rgb}{0.75,0.0,0.0}
\definecolor{nicos-green}{rgb}{0.0,0.75,0.0}
\definecolor{light-gray}{gray}{0.6}
\definecolor{really-light-gray}{gray}{0.8}
\definecolor{sussexg}{rgb}{0.0,0.5,0.5}
\definecolor{sussexp}{rgb}{0.5,0.0,0.5}

\usepackage{dsfont} 
\usepackage[symbol]{footmisc}
\usepackage[utf8]{inputenc}

\usepackage{nicefrac}

\usepackage{multicol}
\PassOptionsToPackage{dvipsnames}{xcolor}
\usepackage{tikz,xcolor}

\usetikzlibrary{shapes,arrows}
\usetikzlibrary{hobby}
\usetikzlibrary{decorations.markings}

\newtheorem{theorem}{\color{my-blue}{\sc Theorem}}[section]
\newtheorem{lemma}[theorem]{\color{my-blue} \sc Lemma}
\newtheorem{claim}{\color{my-blue}{\sc Claim}}
\newtheorem{proposition}[theorem]{\color{my-blue} \sc Proposition}

\newtheorem{question}[theorem]{\color{my-blue} \sc Question}

\newtheorem{definition}[theorem]{\color{my-blue} \it Definition}
\numberwithin{equation}{section}
\theoremstyle{remark}
\newtheorem{remark}[theorem]{\color{my-blue} Remark}

\newcommand{\be}{\begin{equation}}
\newcommand{\ee}{\end{equation}}
\newcommand{\nn}{\nonumber}

\newcommand{\eqpd}{\, .}
\newcommand{\eqcom}{\, ,}

\newcommand{\rbrac}[1]{\left(#1\right)} 
\newcommand{\sbrac}[1]{\left[ #1\right]} 
\newcommand{\cbrac}[1]{\left\{ #1\right\}} 

\def\bE{\mathbb{E}}

\def\bP{\mathbb{P}}

\def\bZ{\mathbb{Z}}

\def\cH{\mathcal{H}}

\def\e{\varepsilon}

\def\cE{\mathcal{E}}

 \def\Z{\bZ}

\def\P{\bP}

\usepackage{stackengine}

\def\E{\bE}
\def\P{\bP} 

\definecolor{partcolor1}{rgb}{0.0,0.5,0.0}
\definecolor{partcolor2}{rgb}{0.0,0.5,0.0}

\definecolor{darkgreen}{rgb}{0.0,0.5,0.0}
\definecolor{darkblue}{rgb}{0.5,0.1,0.5}
\definecolor{deepblue}{rgb}{0.25,0.41,0.88}
\definecolor{nicosred}{rgb}{0.65,0.1,0.1}
\definecolor{light-gray}{gray}{0.7}
\allowdisplaybreaks[1]

\definecolor{navy}{rgb}{0.0,0.0,0.6}
\definecolor{purple(x11)}{rgb}{0.63, 0.36, 0.94}

\newcommand{\F}{\mathcal{F}}

\newcommand{\K}{\mathcal{K}}

\newcommand{\hnpw}{\mathcal{H}_{n,p,w}}

\newcommand{\hnw}{\mathscr{H}_{n, w}}

\newcommand{\Zs}{%
\text{\ooalign{\hidewidth\raisebox{0.2ex}{--}\hidewidth\cr$Z$\cr}}%
}

\newcommand{\mc}[1]{\mathcal{#1}}

\RequirePackage{datetime} 
\allowdisplaybreaks
\begin{document}
\usdate
\title[The Sunflower-Free Process]
{The Sunflower-Free Process}

\author{Patrick Bennett}
\address{Patrick Bennett, Western Michigan University, United States}
\email{patrick.bennett@wmich.edu} 
\thanks{The first author's research is supported in part by Simons Foundation Grant \#848648.}

\author{Amanda Priestley}
\address{Amanda Priestley, The University of Texas at Austin, United States}
\email{amandapriestley@utexas.edu}
\date{\today}
\begin{abstract}

An $r$\textit{-sunflower} is a collection of $r$ sets such that the intersection of any two sets in the collection is identical.  We analyze a random process which constructs a $w$-uniform $r$-sunflower free family starting with an empty family and at each step adding a set chosen uniformly at random from all choices that could be added without creating an $r$-sunflower with the previously chosen sets. 
To analyze this process, we extend results of the first author and Bohman \cite{BennettBohmanNoteOnRandom}, who analyzed a general random process which adds one object at a time chosen uniformly at random from all objects that can be added without creating certain forbidden subsets. 
\end{abstract}
\maketitle
\vspace*{-0.6cm}
\section{Introduction} \label{sec:Introduction}
An $r$\textit{-sunflower}, or a \textit{sunflower with $r$ petals}, is a collection of $r$ sets such that the intersection of any two sets in the collection is identical. When $r$ is clear we suppress it and just say \textit{sunflower}. We will refer to this common intersection as the \textit{kernel} of the sunflower. In particular, any collection of $r$ disjoint sets (resp. $r$ identical sets) forms an $r$-sunflower with an empty kernel (resp. a kernel equal to the identical sets). Let $f(w,r)$ denote the least integer so that any $w$-uniform family of $f(w,r)$ sets contains a sunflower with $r$ petals.
Most will be familiar with the famous \textit{sunflower conjecture} of Erd\H{o}s and Rado \cite{erdos1960intersection} which hypothesizes that, $f(w,r)<C^w$ for some constant $C \coloneqq C(r)$, and they proved the bounds  
\[
(r-1)^w<f(w, r) \leq w!(r-1)^w+1.
\] 
Their lower bound follows from the following construction. Let 
$\{X_1, \dots X_w\}$ be a collection of $w$ pairwise disjoint sets, each of size $r-1$. The family $\mathcal{F}$ defined as 

\[
\F \coloneqq  X_1\times X_2 \times \dots \times X_w \eqpd
\]
This family appears to be the most common construction referenced when discussing known examples of large sunflower-free families \cite{AlweissLovettWuZhang, KostochkaSurvey}. While there are quite a few meticulous constructions of Abbott and coauthors for special cases of $w$ and $r$, (see, e.g.; \cite{abbott1969combinatorial, abbott1972intersection, abbott1974finite,abbott1992set}) until the present work, to our knowledge there has been no other analysis of a general algorithm for constructing a $w$-uniform $r$-sunflower-free family.\\

In this work, we consider the following random process which we will call the \textit{$r$-sunflower-free process}, which generates a $w$-uniform $r$-\textit{sunflower-free} hypergraph. Going forward, let $\mathcal{K}_{n,w}$ denote the complete $w$-uniform hypergraph containing all possible $\binom{n}{w}$ sets. Let $\mathcal{L}_0\coloneqq \mathcal{K}_{n,w}$, and in addition, let $\mathcal{H}_0\coloneqq \emptyset$. In step $i$ of the process for $i\geq 1$,  an edge $e$ is chosen from $\mathcal{L}_{i-1}$ each with uniform probability, and added to $\mathcal{H}_{i-1}$ to form $\mathcal{H}_i$ so long as adding $e$ to $\mathcal{H}_{i-1}$ does not construct an $r$ sunflower in $\mathcal{H}_i$. Whether or not the edge is included, $e$ is deleted from $\mathcal{L}_{i-1}$ such that $\mathcal{L}_i\coloneqq \mathcal{L}_{i-1} \setminus e$. The algorithm continues until the final step $t$ such that $\mathcal{L}_t = \emptyset$. An equivalent version of this algorithm just chooses $e$ uniformly at random from all edges that would be accepted. This algorithm was introduced in an old paper of Abbott and Exoo \cite{abbott1992set}, and a
modified version (including an additional backtracking step) was used to construct the largest known example of a $3$-uniform $4$-sunflower-free family. The algorithm was used only as a method of constructing small concrete examples by way of computer, and has yet to be analyzed in its own right.\\

Based on the known constructions for good lower bounds on $f(w, r)$, it seems likely that very large sunflower-free collections must be ``structured,'' in the sense that these constructions tend to exploit specific properties of set systems. In this paper we will not come close to improving the best known lower bound, but rather we will address the following question:
\begin{question}
     What does a ``typical, non-structured'' sunflower-free collection look like? How large can it be?
\end{question} 
\noindent One simple way to interpret this question is to generate a collection of $w$-sets contained in a universe of $n$ points, taking each $w$-set independently with probability $p$, where $p$ is small enough that our random collection is unlikely to contain an $r$-sunflower. In particular, 
following the notation of \cite{frieze2016introduction}, let $\mathbb{H}_{n,p,w}$ denote a random hypergraph where each of the possible $\binom{n}{w}$ edges is chosen independently with probability $p$. Then we have the following.
\begin{proposition}\label{prop:Hnpw}  
    Fix $r \in \mathbb{N}$ and let  $w = n^\alpha$ for some fixed $\alpha \in \left(\frac{1}{2}, 1\right)$. In addition, let 
    \[
    N:= \binom nw, \qquad D \coloneqq \sum\limits_{s=0}^{w-1} \frac{w!(n-w)!}{(r-1)! s![(w-s)!]^{r} [(n-rw+(r-1)s]!},
    \] where $N$ is the number of edges, and $D$ denotes the number of sunflowers containing a fixed edge. Then,
    \[
    \Pr\left(\ \mathbb{H}_{n,p,w} \mbox{ contains an } r\mbox{ sunflower }\right)\to \begin{cases}
        0 &\mbox{ if } p=o\left({ND}\right)^{-\frac{1}{r}}\\
        1&\mbox{ if } p= \omega\left({ND}\right)^{-\frac{1}{r}}\
    \end{cases}
    \]
\end{proposition}
\noindent
Note that the case where $w=n^\alpha$ for $\alpha < 1/2$ is less interesting, since then two random edges probably do not intersect at all, and so a random set of $r$ edges is likely a sunflower (matching) already. Even for $\alpha = 1/2$ there is a constant probability that two random edges intersect, and so we probably have a matching on $r$ edges when the total number of edges is a large constant. \\

If one wants a larger random collection than the one given by the above proposition, one can use the alterations method. To apply this method we would again take a random collection of density $p$, but allow $p$ to be large enough so that our random collection probably does contain some (not too many) $r$-sunflowers. We then remove some sets from our collection to get rid of any $r$-sunflowers. In this work we will analyze a process that produces an even larger collection than the alterations method.\\

On the other hand, one might consider a hypergraph, $\mathbb{H}_{n,w,m}$ chosen uniformly from the set of all $w$-uniform hypergraphs on $n$ vertices, with $m$ edges. In this case, using the celebrated result of Frankston, Kahn, Narayanan, and Park \cite{FKNP}, one can say the following.
\begin{proposition}\label{prop:SpreadnessThreshold} 
Let $r = r(n) \to \infty$. If $m \geq Ke^{2w-t}r n^{t}\log r$, for some absolute constant $K$, then with high probability $\mathbb{H}_{n,m,w}$ contains a sunflower with $r$ petals and a kernel of size $t$. 
\end{proposition}
Note that in the above proposition, we must let $r$ grow with the number of points in the universe, $n$. This is only point in the paper in which we consider this regime of $r$, and in what follows (and what came before) we think of $r\in \mathbb{N}$ as being fixed.
\\

For our main result, however, we will analyze the sunflower-free process in detail by adopting the perspective that a sunflower-free collection of sets can be viewed as an independent set in an appropriately defined hypergraph.
An \textit{independent set} in $\mathscr{H}$ is a subset of the vertices containing no edge of $\mathscr{H}$. For the the sunflower-free problem on a universe $U$ of $n$ points, the hypergraph of interest is denoted $\hnw$. $\hnw$ has vertex set $\mathscr{V}_{n,w} = \binom{U}{w}$, the collection of all 
sets of size $w$ in $U$. $\hnw$ has edge set  $\mathscr{E}_{w,r}$, which is the  collection of all sets of $r$-sunflowers (i.e. sets of $r$ many sets of size $w$ that form a sunflower).
Going forward, we will call $\hnw$ the \textit{sunflower hypergraph}. Thus, an independent set in $\hnw$ is just a $w$-uniform $r$-sunflower-free collection.\\

 The random greedy independent set process, or just the \textit{independent process}, for a hypergraph of  $\mathscr{H}$, is defined as follows. We use similar notation to that used in \cite{BennettBohmanNoteOnRandom} for ease of translation. We begin with $\mathscr{H}(0) \coloneqq \mathscr{H}$, and in the $(i+1)^{th}$ step, a vertex $v$ from the set of \textit{open vertices} $V(i)$, (where $V(0) \coloneqq V(\mathscr{H})$) is chosen to be added to the independent set $I(i)$ where $I(0)\coloneqq \emptyset$. After $v$ is added to $I(i)$ to form  $I(i+1)$, the following updates are performed to obtain $\mathscr{H}(i).$
\begin{enumerate}
    \item \label{step1} The \textit{closed vertices} are removed from $V(i)$ to form $V(i+1)$. That is, the vertex that was chosen to be added to the independent set in step $i+1$, $v_{i+1}$, is removed, along with every vertex $u\in V(i)$ such that $\{v,u\}$ is an edge of $\mathscr{H}(i)$.
    \item \label{step2} The edge set of $\mathscr{H}(i)$ is updated by removing the vertex $v_{i+1}$ from every edge containing $v_{i+1}$  and at least two other vertices, and every edge containing the an edge of the form $\{v_{i+1}, u\} \in \mathscr{H}(i)$  
    is also removed. 
    \item \label{step3} In general, by convention, for any edge $e'\in \mathscr{H}(i) $ such that there exists an edge $e\in \mathscr{H}(i) $ with $ e \subseteq e'$, $e'$ is also removed.
\end{enumerate}

The independent process has been analyzed for many specific hypergraphs and in general for ``nice'' hypergraphs. The most well-known example is the triangle-free process, first analyzed by Bohman \cite{bohmantriangle} and then in greater detail by Bohman and Keevash \cite{bohmankeevashtriangle} and independently by Fiz Pontiveros, Griffiths, and Morris \cite{pontiverosgriffithsmorris}. This process is a randomized graph algorithm which gave the best known lower bound on the Ramsey numbers $R(3, t)$ until it was very recently improved by Campos, Jenssen, Michelen and Sahasrabudhe \cite{CJMS} (though the new lower bound still uses a similar randomized construction together with new ideas). More generally, the $H$-free process was analyzed by Bohman and Keevash \cite{bohmankeevashhfree} and gives the best known lower bound for $R(s, t)$ for fixed $s\geq 5$ (until recently it was also the case for $s=4$, but the lower bound on $R(4, t)$ has now been improved by Mattheus and Verstraete \cite{mattheusverstraete} using completely different ideas). To encode the $H$-free process as an instance of the independent process, we define the following hypergraph $\mathscr{H}_H$. The $H$-free process is then just the independent process on $\mathscr{H}_H$.
\begin{definition}\label{def:hfree}
    Let $H$ be a fixed graph. Then $\mathscr{H}_H$ is the hypergraph with vertex set $E(K_n)$, where each hyperedge of $\mathscr{H}_H$ is the edge set of an isomorphic copy of $H$.
\end{definition}

In a note of the first author and Bohman \cite{BennettBohmanNoteOnRandom} a general lower bound was given on the size of an independent set, $I$,  given by the random greedy independent set process under certain assumptions about the codegrees in the hypergraph $\mathscr{H}$. Unfortunately there are some interesting hypergraphs, including $\hnw$, which do not satisfy the assumptions in \cite{BennettBohmanNoteOnRandom}. In this paper we modify the analysis to tolerate relaxed assumptions. In particular, this will allow us to analyze the sunflower-free process.
\\

The method we will use is the {\em differential equation method}, which is a collection of techniques and tools useful for analyzing random processes which select one object at each step. For these processes, the selection is generally random, but may depend on previous steps. Typically (as will be the case for us), we are forbidden from making certain selections based on the result of previous steps, but the distribution at any given step is uniform over all non-forbidden possibilities. The method is strongly associated with Wormald, who played a crucial role in developing it and finding many applications \cite{wormald}. The unfamiliar reader can also refer to a gentle introduction to the method by the first author and Dudek \cite{bennettdudekgentle}. The differential equation method is the approach used in the papers \cite{bohmantriangle, bohmankeevashtriangle, pontiverosgriffithsmorris, bohmankeevashhfree} mentioned previously. \\

Recall that $V(i)$ denotes the set of \textit{open} vertices in step $i$, with $|V(0)| =N$. We say that $\mathscr{H}$ is \textit{$D$-regular} if each vertex $v \in V(\mathscr{H})$ is in exactly $D$ edges. We let $d_A$ denote the \textit{degree of a subset} $A \subseteq N$ of the vertices, meaning the number of hyperedges in $\mathscr{H}$ containing $A$. Furthermore let $\Delta_{\ell}(\mathscr{H})$ denote the \textit{maximum degree of an $\ell$-subset of the vertices.} That is, 
\[
\Delta_{\ell}(\mathscr{H}) = \max\{d_A : A \subseteq N, |A|=\ell\} \eqpd
\]
We let the \textit{$(r-1)$-codegree} of a pair of distinct vertices $v, v^{\prime}$ be the number of subsets  $S \subseteq V$ of size $r-1$ such that $S\cap \{v, v'\} =\emptyset$, and $\{v\} \cup S$ and $\{v'\} \cup S$ are both edges of  $\mathscr{H}$.

Finally, we let $\Gamma(\mathscr{H})$ denote the maximum $(r-1)$-codegree of over all pairs of vertices $v, v'$.\\

As one of our main theorems is a modification of the theorem of the first author and Bohman \cite{BennettBohmanNoteOnRandom}, we restate it here for convenience. 
\begin{theorem}[Theorem 1.1 in \cite{BennettBohmanNoteOnRandom}]\label{thm:bennettBohman} Let $r\geq 3$ and $\e>0$ be fixed. Let $\mathcal{\mathscr{H}}$ be a $r$-uniform, $D$-regular hypergraph on $N$ vertices such that $D>N^\e$. If
$$
\Delta_{\ell}(\mathscr{H})<D^{\frac{r-\ell}{r-1}-\e} \quad \text { for } \ell=2, \ldots, r-1
$$
and $\Gamma(\mathscr{H})<D^{1-\e}$ then the random greedy independent set algorithm produces an independent set $I$ in $\mathscr{H}$ with
$$
|I|=\Omega\left(N \cdot\left(\frac{\log N}{D}\right)^{\frac{1}{r-1}}\right)
$$
with probability $1-\exp \left\{-N^{\Omega(1)}\right\}$.
\end{theorem}

\begin{remark}Note that under relaxed assumptions (assuming only that the maximum degree is $D$), one can prove the existence of an independent set of size $\Omega(ND^{-\frac{1}{r-1}})$ using either the alterations method or the Lov\'asz Local Lemma. Theorem \ref{thm:bennettBohman} gives a logarithmic improvement.
\end{remark}

 In this work, we relax the assumptions of Theorem \ref{thm:bennettBohman} in several ways. For example, we allow that the hypergraph $\mathscr{H}$ no longer be $r$-uniform, allowing some smaller edges.  
We also allow for some variation in the degrees of vertices in $\mathscr{H}$, only requiring them to be $(1+o(1))D$. We allow the set degrees $\Delta_\ell(\mathscr{H})$ to be larger.
And perhaps most importantly, due to our interest in the sunflower hypergraph, we are in need of a theorem which significantly relaxes the assumption on the $(r-1)$-codegree of $\mathscr{H}$. We will allow the $(r-1)$-codegrees to be larger, and also we will allow there to be some pairs of vertices for which we do not assume anything at all about the $(r-1)$-codegree. This is important, since we will later show that when looking at the sunflower hypergraph $\hnw$ in the regime where $w = n^\alpha$ and $\frac{1}{2}< \alpha <1$, the $(r-1)$-codegree of two vertices $v, v^{\prime} \in \mathscr{V}_{w,r} $ corresponding to hyperedges $e_{\K}, e_{\K}^{\prime} \in \K_{n,w}$ is maximized when $e_{\K}$ and $e_{\K}^{\prime}$ have large intersection. In fact, when $|e_{\K} \cap e_{\K}'| = w-1$ the $(r-1)$-codegree is nearly $D$. We will see that in this regime, most sunflowers in $\K_{n,w}$ will have relatively small kernels, so the sets forming a sunflower typically barely intersect and mostly avoid each other. Therefore if $e_{\K}$ and $e_{\K}^{\prime}$ have a large intersection, there will be many sets in $\K_{n,w}$ that mostly avoid both of them, and could easily form a sunflower with either one. \\

We will need to assume this phenomenon (pairs with high $(r-1)$-codegree) to be relatively rare. Toward this end, let us define the set of \textit{bad pairs} as follows. The definition uses a parameter $\phi$ which can be thought of as small.

\begin{definition}[Bad Pairs]\label{def:BadVertices}
   Vertices $v, v^{\prime}\in V
   \left(\mathscr{H}\right)$ form a bad pair if their $(r-1)$-codegree is at least $\phi D$. For a fixed vertex $v$, the set of vertices $v^{\prime}$ which form a bad pair with $v$ will be denoted by $B(v)$. 
\end{definition}

Roughly speaking, bad pairs are highly correlated. Supposing that $v, v'$ is a bad pair and $v'$ is chosen to go in our independent set, then $v$ becomes significantly more likely to end up in the independent set. Indeed, there are many $(r-1)$-sets $S$ such that $\{v\} \cup S$ and $\{v'\} \cup S$ are both hyperedges. When we choose $v'$, due to step \eqref{step2} in our description of the process, we remove $v'$ from $\{v'\} \cup S$, leaving just $S$. But since $S \subseteq \{v\} \cup S$, in step \eqref{step3} we delete the hyperedge   $\{v\} \cup S$. Since there are many such sets $S$, the degree of $v$ drops by a lot in the step where we pick $v'$. This means $v$ becomes less likely to ever be closed, meaning that it is more likely to eventually be chosen. We will discuss more specific examples of this phenomenon in Sections \ref{sec:sumfree} and \ref{sec:blowup}.

To state our next theorem we will need the following definitions. For a (possibly non-uniform) hypergraph $\mathscr{H}$, we say that $\mathscr{H}$ is \textit{$r$-bounded} if every edge has size at most $r$. We let $\mathscr{H}^{(s)}$ denote the $s$-uniform hypergraph consisting of all edges of size $s$ from $\mathscr{H}$.

\begin{definition}[Degrees of Sets] \label{def:DegreesOfSets}For a set of vertices $A$, let $d_{A \uparrow b}$ be the number of edges of size $b$ containing $A$ in $\mathscr{H}$. If $A =\{v\}$ is a singleton then we just write $d_{v \uparrow b}$.
\end{definition}

\begin{definition}[Bad Degree of Sets]
   For a set of vertices $A$, let $d^{\prime\prime}_{A \uparrow b}$ be the number of edges of size $b\ge |A|+2$ containing $A$ and a pair of bad vertices $v, v' \notin A$.
\end{definition}

Finally, we say that a sequence of events $\mc{E}_N$ occurs \textit{asymptotically almost surely} or \textit{a.a.s.} if  $\Pr[\mc{E}_N] \rightarrow 1$ as $N \rightarrow \infty$. With these definitions in mind, we can state our main theorem as follows.

\begin{theorem}\label{thm:ind} Fix an integer $r\ge 3$, and positive functions $\phi=\phi(N) \le \log^{-300r}N$ and  $D=D(N)\ge \phi^{-r}$. Suppose $\mathscr{H}$ is an $r$-bounded hypergraph on $N$ vertices such that  

\begin{enumerate}
    \item \label{cond:almostregular} for every vertex $v$, we have $d_{v \uparrow r} \in [(1-\phi)D, (1+\phi)D]$, 
    \item \label{cond:maxLdegree}   for  $2 \le \ell < k \le r$ we have $
\Delta_{\ell}(\mathscr{H}^{(k)}) \le \phi D^{\frac{k-\ell}{r-1}}$,
\item \label{cond:max1degree}  for  $2 \le k \le r-1$ we have $
\Delta_{1}(\mathscr{H}^{(k)}) \le \phi D^{\frac{k-1}{r-1}}$,
\item \label{cond:Bbound}  for all $v \in V(\mathscr{H})$ we have $|B(v)| \le \phi D^{\frac{1}{r-1}}$, and
\item \label{cond:ddubprime}  for all $A \subseteq V(\mathscr{H})$ with $|A|\ge 1$ and  $|A|+2 \le b \le r$ we have $d^{\prime\prime}_{A \uparrow b}(\mathscr{H}) \le \phi D^{\frac{b-|A|}{r-1}}$.
\end{enumerate} 
 Then the random greedy independent set algorithm produces an independent set $I$ in $\mathscr{H}$ with
$$
|I|=
\Omega\left(N \cdot\left(\frac{\log 1/\phi}{D}\right)^{\frac{1}{r-1}}\right) 
$$
asymptotically almost surely.
\end{theorem}

\begin{remark}
    The hypotheses of Theorem \ref{thm:ind} imply that $\mathscr{H}$ is not too dense. Specifically, counting $|\mathscr{H}|$ in two ways we have
$$
(1-\phi)\frac{N D}{r}\leq |E(\mathscr{H})| \leq \frac{1}{\binom{r}{2}}\binom{N}{2}\phi D^{\frac{r-2}{r-1}} .
$$
This implies that we have
\be
\begin{split}\label{eq:LowerBoundonN}
    N=\Omega\left(D^{\frac{1}{r-1}}\phi^{-1}\right), \qquad D = O\rbrac{\phi^{r-1} N^{r-1}} \eqpd
\end{split}
\ee
\end{remark}

\bigskip
A similar result was recently obtained by Guo and Warnke \cite{GW}, so we briefly discuss the similarities and differences. Instead of analyzing the random greedy process, Guo and Warnke \cite{GW} consider a semi-random ``nibble'' version. In other words, they consider a random process in which one ``step'' involves choosing many potential vertices to go into the independent set, and possibly discarding some of them which would cause a problem. The main theorem in \cite{GW} requires only an upper bound on vertex degrees, whereas our Theorem \ref{thm:ind} requires almost-regularity. Their main theorem requires the hypergraph to be $r$-uniform, while our Theorem \ref{thm:ind} does not. Their main theorem also has conditions very similar to our Condition \eqref{cond:maxLdegree}, as well as a stronger (compared to Theorem \eqref{thm:ind}) assumption on the $(r-1)$-codegrees. Roughly speaking, in the language of our paper, the main theorem in \cite{GW} assumes there are no bad pairs of vertices. For that reason, the result of Guo and Warnke could not be directly applied to the hypergraph $\hnw$ to obtain a sunflower-free collection.\\

As mentioned, one of the central motivations behind introducing some of the conditions in Theorem \ref{thm:ind}, is in order to apply it to the sunflower-free process. In doing so, we obtain the following. (Note that here $r$ is held constant while $w$ grows with $n$.)

\begin{theorem}\label{thm:SunflowerTheorem} Fix $r\geq 3$ and  let $w=n^\alpha$ for some fixed $\alpha \in \rbrac{\frac12, 1}$. Let $N = \binom{n}{w}$ and $D$ be defined as in \eqref{eq:SF-Ddef}. The sunflower-free process run on a universe of size $n$ produces a $w$-uniform $r$-sunflower-free  family of size
\[
  \Omega \rbrac{
 (w^2/n)^\frac{1}{r-1}
 ND^{-\frac{1}{r-1}}}
\]
asymptotically almost surely.

\end{theorem}

Let us define the following parameter for the sunflower-free process,
\[
\kappa(n,w) \coloneqq \left(\frac{w^2}{nD}\right)^{\frac{1}{r-1}}\eqpd
\]
Theorem \ref{thm:SunflowerTheorem} then says that the sunflower-free process is able to construct a hypergraph with a $\kappa(n,w)$ fraction of the total number of possible edges, asymptotically almost surely, before terminating.  Theorem \ref{thm:SunflowerTheorem} gives an improvement by the factor $(w^2 / n)^\frac{1}{r-1}$ over the alterations method or Lov\'asz Local Lemma (see the remark after Theorem \ref{thm:bennettBohman}).
  Of course Theorem \ref{thm:SunflowerTheorem} only gives a lower bound on the final size of the sunflower-free family, so one might hope that the truth is significantly larger. However, the bound in Theorem \ref{thm:SunflowerTheorem} follows from the bound in Theorem \ref{thm:ind}, and it is known that there exist instances (e.g. when $\mathscr{H}$ is a random hypergraph) where the largest independent set in $\mathscr{H}$ matches the lower bound up to a constant factor. Furthermore, there exist hypergraphs $\mathscr{H}$ that do have much larger independent sets, but for which it is known that the random greedy algorithm terminates long before finding one. For example, Mantel's theorem gives us that the largest independent set in $\mathscr{H}_{K_3}$ (see Definition \ref{def:hfree}) has size $\lfloor n^2/4 \rfloor$. But the independent process on $\mathscr{H}_{K_3}$, i.e.~the triangle-free process, terminates after only $O(n^{3/2} \log^{1/2} n) $ steps a.a.s. \cite{bohmankeevashtriangle, pontiverosgriffithsmorris} (note that this upper bound number of steps matches the lower bound given by Theorem \ref{thm:ind}). Bohman and Keevash \cite{bohmankeevashhfree} conjecture that in general, the lower bound they proved for the $H$-free process (which matches what we get from Theorem \ref{thm:ind}) is correct up to a constant factor. This conjecture has
been verified in a few more cases, including the $K_4$-free process \cite{Wa3,Wz2}
and the $ C_\ell$-free process for all $ \ell \ge 4 $ \cite{P2,P3,Wa2}. In light of all this, it seems more likely that the lower bound in Theorem \ref{thm:SunflowerTheorem} is not far from the truth. If that is the case, then the family produced by the sunflower-free process is small compared to other known constructions. Indeed, in some of the best known constructions, the authors make use of the fact that the precise answer is known for very small cases of $r$ and $w$ (e.g., \cite{AbbottGardner}) and then use structural properties of set systems to iteratively combine small sunflower-free families on disjoint ground sets, to form larger ones \cite{abbott1992set,abbott1972intersection}. However, it's the case that the sunflower-free family given by Theorem \ref{thm:SunflowerTheorem} is still relatively large for a ``random-looking'' sunflower-free family. To be more precise, by ``random-looking,'' we mean that our sunflower-free family has \textit{pseudorandom} properties, i.e.~it shares some properties with a uniformly random family of $w$-sets of the same density. In this paper we do not go out of our way to prove these pseudorandom properties, but some are established as part of the proof of Theorem \ref{thm:ind}. For example, the proof implies that in our final $r$-sunflower-free family, certain substructure counts are pseudorandom (approximately the same as they would be in a uniform random family of the same density). In \cite{BennettBohmanNoteOnRandom}, the first author and Bohman proved that under the assumptions of Theorem \ref{thm:bennettBohman}, a wide variety of substructure counts in the final independent set are pseudorandom. This ought to also be the case in the setting of Theorem \ref{thm:ind}, and the authors plan to pursue some preliminary work they have already done on such a result. It is worth noting that Guo and Warnke \cite{GW} proved that their semi-random process produces an independent set with pseudorandom properties.\\ 

When considering the bound in Theorem \ref{thm:SunflowerTheorem}, we know due to the sunflower lemma of Erd\H{o}s and Rado, that any family containing at least $w!(r-1)^{w}$ sets must contain a sunflower. This bound was improved in the breakthrough result of Alweiss et al. \cite{AlweissLovettWuZhang}, which was quickly followed by improvements by Rao \cite{Rao2020Coding}, and Bell et al. \cite{BellNoteOnSunflowers} to give that $(\alpha r\log(w)) )^w$ sets are sufficient for guaranteeing the existence of a sunflower, for some absolute constant $\alpha$. What this implies in our setting is that, when $n$ is sufficiently large such that $\binom{n}{w}$ is much larger than $(\alpha r\log(w)) )^w$, adding additional points to the universe won't help the algorithm to construct a larger $r$-sunflower-free  family. In particular, for fixed $w$ it makes sense for $\kappa(n,w)$ to increase as a function of $n$ up to some threshold ${n^{\ast}}$, after which, it should tend to $0$ as $n \rightarrow \infty$. Thus, characterizing this threshold behavior as a function of $n$, is an interesting question in its own right.

\subsection{Discussion of the conditions in Theorem \ref{thm:ind}}

Here we discuss why we need something resembling the assumptions of Theorem \ref{thm:ind}, and specify a couple of additional examples that satisfy those conditions, but do not satisfy the conditions of Theorem \ref{thm:bennettBohman}.\\

Our analysis of the process is fundamentally motivated by tracking the number of open vertices, and so we must also keep track of how many vertices become closed. The mechanism that causes a vertex $v$ to become closed is an edge $e$ that contains $v$ where all other vertices in $e$ are in the independent set $I$. This does not happen all at once: in the beginning none of the vertices of $e$ are in $I$, and as the process runs we put vertices of $e$ into $I$. To keep track of this phenomenon, for each vertex $v$ we will need to estimate $d_{\ell}(v)$, which is the number of edges containing $\ell$ open vertices including $v$, with the rest of the vertices in the edge belonging to $I$. So we will keep track of a large system of random variables including both the number of open vertices and these variables $d_\ell(v)$.
For each of these random variables, we will show that it is a.a.s.~within a relatively small error of some deterministic function, called the \textit{trajectory} of that random variable. (Technically, we will only do this for {\em almost} all the variables $d_\ell(v)$, allowing for the possibility that a few of them deviate significantly from their trajectories)\\

The differential equation method naturally lends itself to treating all vertices the same, and so it is hard to imagine analyzing the process without something like Condition \eqref{cond:almostregular}.  For the differential equation method, we also need a \textit{boundedness hypothesis} for the variables we track, meaning that we need a bound on how large the change can be over one step of the process. Condition \eqref{cond:maxLdegree} is reminiscent of the ``spreadness'' condition in \cite{AlweissLovettWuZhang} which was an integral ingredient in achieving an improved upper bound on the theorem of Erd\H{o}s and Rado. A hypergraph $\mathscr{H}$ is $\kappa$-spread if $\Delta_\ell(\mathscr{H}) \le \kappa^\ell |\mathscr{H}|$ for each $\ell$. Roughly speaking, spreadness translates to the assumption that the degree of a subset of vertices decreases by a factor of $\kappa$ as the set increases in size by a single element. In our case, Condition \eqref{cond:maxLdegree} helps with the boundedness hypothesis for our variables $d_\ell(v)$. For example, in one step the maximum possible increase in $d_{r-1}(v)$ is $\Delta_2(\mathscr{H}^{(r)})$. It is worth noting that, in the absence of the imposition of a spread-like condition, the independent process can behave very differently. Indeed, Picollelli \cite{P4} studied the $H$-free process where $H = K_4^-$ is formed by removing one edge from $K_4$. This process is of course equivalent to the independent process on the hypergraph $\mathscr{H}_{K_4^-}$, which does not satisfy Condition \eqref{cond:maxLdegree}. While Picollelli was still able to analyze the process, the trajectories of all the random variables significantly deviated from what we will show is true of the independent process under the assumptions of Theorem \ref{thm:ind}.  Condition \eqref{cond:max1degree} is there to guarantee that even though we allow for some smaller edges, they will have a negligible effect, in the sense that most vertex closures should be from edges that originally had size $r$. For the applications we have in mind (e.g.~the sum-free process discussed in the next section), we view the smaller edges as degenerate cases which should not affect much. Condition \eqref{cond:Bbound} is also related to the boundedness hypothesis. For example, when we choose $v_i$ to go into our independent set, for an arbitrary vertex $v$, $d_r(v)$ could decrease by the $(r-1)$-codegree of $v$ with $v_i$  due to Step \eqref{step3}. We actually do allow some pairs of vertices to have large $(r-1)$-codegree, but we must put some limit on how many such pairs there are. If $u, v$ have high $(r-1)$-codegree and $u$ is chosen to go in the independent set, then we are forced to give up on tracking $d_\ell(v)$, which might have drastically changed in that one step. From that point on we only use a crude upper bound on $d_\ell(v)$. We manage to analyze the process so long as we can still keep track of $d_\ell(v)$ for almost all open $v$. Condition \eqref{cond:Bbound} allows us to bound the number of $v$ for which we can no longer track $d_\ell(v)$. Condition \eqref{cond:ddubprime} avoids a certain problem in tracking $d_\ell(v)$. Namely, suppose that many edges contained $v$ as well as a bad pair of vertices. As the process runs, we would like to keep track of how these edges contribute to $d_{\ell}(v)$ for each $\ell \ge 2$. These contributions occur due to these edges having vertices chosen to go in the independent set, but if we choose a member of a bad pair, the other member starts to behave unpredictably. Our analysis can handle only a few such contributions to $d_{\ell}(v)$, and Condition \eqref{cond:ddubprime} gives us that these contributions will not become pathological. 

\subsubsection{The sum-free process} \label{sec:sumfree}
Another random process which can also be characterized as finding a large independent set in a hypergraph is that of the \textit{sum-free process}. The sum-free process on $\Z_{2n}$ iteratively constructs a set $S \subseteq \Z_{2n}$, starting with $S=\emptyset$, and in each step adds a uniform random element of $\Z_{2n}$ so long as  $S$ maintains the property that there exists no choice of $a,b,c \in S$, such that $a+b=c$ is satisfied. In this case, this is equivalent to the random greedy independent set process run on  the hypergraph $\mathscr{H}$ with vertex set $\mathscr{V}=\Z_{2n}$, and edge set $\mathscr{E}$ which is the family of all solutions $\{a,b,c\}$ where $a+b=c$ (or the set of all \textit{Schur triples)}. Of course there are a few ``triples'' where not all of $a, b, c$ are distinct, and we avoid those solutions as well. 
In \cite{B20}, the first author showed the following.
\begin{theorem}[Theorem 3 of \cite{B20}]\label{thm:sumfree}
Asymptotically almost surely, the sum-free process run on $\mathbb{Z}_{2 n}$ produces a set $S$ of size
\[
|S| \geq \frac{1}{6} n^{1 / 2} \log ^{1 / 2} n .
\]
\end{theorem}
\noindent
The proof of this theorem similarly required modifying the proof of Theorem \ref{thm:bennettBohman} in order to tolerate a modest number of vertices with large codegree. In fact, we can observe that our new Theorem \ref{thm:ind} implies Theorem \ref{thm:sumfree} (except ``1/6'' is replaced by some unspecified constant). \\

To apply Theorem \ref{thm:ind} to the sum-free process, we can let $r=3$ and $N=2n-1$ (we do not include the vertex $0$ since 0+0=0). We let $D = 3n$ and say $\phi=n^{-1/3}$. To check Condition \eqref{cond:almostregular}, note that every vertex $v$ is in $2n+O(1)$ solutions to $v+b=c$ (only counting cases where $v, b, c$ are distinct) and $n+O(1)$ solutions to $a+b=v$, and so $d_{v \uparrow 3} = 3n+O(1)$. For Condition \eqref{cond:maxLdegree}, note that the only relevant case is $\ell=2, k=3$, and $\Delta_2(\mathscr{H}^{(3)}) = O(1)$. For Condition \eqref{cond:max1degree} the relevant case is $k=2$ and since edges of size 2 occur when we have $a+b=c$ and $a, b, c$ are not all distinct, we have $\Delta_1(\mathscr{H}^{(2)}) = O(1)$. For Condition \eqref{cond:Bbound}, consider a bad pair $v, v'$, so there are at least $\phi D = \Omega(n^{2/3})$ pairs $x, y$ such that both $xyv$ and $xyv'$ are Schur triples. For $xyv$ to be a Schur triple we need $v \in \{x+y, x-y, y-x\}$ and likewise we need $v'\in \{x+y, x-y, y-x\}$. Furthermore we assume $v \neq v'$ for a bad pair. In most cases, e.g.~when $v=x+y$ and $v'=x-y$, there are only $O(1)$ solutions for $x, y$. The only case where we could possibly get more solutions is when $v=x-y, v'=y-x$ which of course can only happen if $v'=-v$. So if $v' \neq -v$ then they have 2-codegree $O(1)$. Each $v$ is in exactly one bad pair $v, -v$, and $|B(v)| =1$, verifying Condition \eqref{cond:Bbound}. Finally, for Condition \eqref{cond:ddubprime} the relevant case is $|A|=1, b=3$. Since bad pairs are all of the form $x, -x$, any fixed vertex $v\neq 0$ can be in $O(1)$ Schur triples consisting of $v$ and a bad pair, verifying Condition \eqref{cond:ddubprime}.\\

As we showed above, the bad pairs in this hypergraph $\mathscr{H}$ are precisely the pairs $v, -v$. As we discussed earlier, bad pairs are highly correlated. We can see this phenomenon illustrated in the sum-free process. Indeed, there are $O(n)$ bad pairs, so if we run the process for $O(n^{1/2} \log^{1/2}n)$ steps, one would guess (based on the heuristic assumption that our independent set resembles a uniform random set) that our independent set contains $O(\log n)$ bad pairs $v, -v$. To the contrary, in \cite{B20} the first author proved that the number of bad pairs in the independent set grows like a power of $n$.

\subsubsection{The independent process on a blowup}\label{sec:blowup}

Let $\mathscr{H}$ be a $D$-regualar hypergraph on $N$ vertices, and assume it satisfies the assumptions of Theorem \ref{thm:bennettBohman}.  The \textit{$b$-blowup} of $\mathscr{H}$ is a hypergraph $\mathscr{H}'$ whose vertex set is $\{(v, i), v \in V(\mathscr{H}), 1 \le i \le b\}$, so we duplicate each vertex of $\mathscr{H}$ $b$ times. For each edge $\{v_1, \ldots, v_r\}$ in $\mathscr{H}$ there are $b^r$ edges $\{(v_1, i_1), \ldots, (v_r, i_r)\}$ in $\mathscr{H}'$ (each $i_j$ ranges from $1$ to $b$). Thus, $\mathscr{H}'$ is $r$-uniform on $N'=bN$ vertices, and $D'$-regular where $D'=b^{r-1}D$.\\

Let $\mathscr{H}'$ be the 2-blowup of $\mathscr{H}$ ($b=2$ here for simplicity). For $\mathscr{H}'$, we see that the bad pairs consist of duplicates $(v, 1), (v,2)$ of the same vertex from $\mathscr{H}$. In fact, the $(r-1)$-codegree of $(v, 1), (v,2)$ is exactly $D'$. If the independent process chooses $(v, 1)$ to go into the independent set, this immediately removes all edges containing $(v, 2)$. Thus $(v, 2)$ will remain open until it is inevitably also chosen to be in the independent set. Likewise, if $(v,1)$ gets closed then $(v, 2)$ gets closed simultaneously. Thus, the fate of $(v, 1)$ (i.e.~whether or not it is in our final independent set) is exactly the same as $(v, 2)$. This is the most extreme possible case illustrating the correlation between a bad pair of vertices.  We also mention this example because even if $\mathscr{H}$ satisfies the assumptions of Theorem \ref{thm:bennettBohman}, $\mathscr{H}'$ does not. However $\mathscr{H}'$ still satisfies the conditions of Theorem \ref{thm:ind}.\\

\subsection{Notation and conventions}
In this paper we use standard asymptotic notation $O(-), \Omega(-), \Theta(-)$. We use the notation $a\approx b$ to denote that $a$ is asymptotically equivalent to $b$. I.e., $\frac{a}{b} \to 1$  as $n \to \infty$.
 We also use expressions with the symbol ``$\pm$'' to represent intervals. For example if $b \ge 0$ then we may write $a \pm b$ to represent $[a-b, a+b]$. If in addition $c \ge b$ then we may write $a \pm b \subseteq a \pm c$.

\section{The good event}  
 Like in \cite{BennettBohmanNoteOnRandom}, our random greedy independent set should behave asymptotically like a binomial random set (except for the fact that a binomial random set may not be an independent set in $\cH$). 
Recall that we use the convention that, if we arrive at a step in the process $i$ such that $\mathcal{H}_i$ that has edges $e, e^{\prime}$ such that $e \subseteq e^{\prime}$ then we remove the larger edge $e^{\prime}$ from $\mathcal{H}$. This choice ends up having no impact on the process as the presence of $e$ ensures that we never have $e^{\prime} \subseteq I(j)$. Furthermore, we use the notation $\Delta X\coloneqq X(i+1)-X(i)$ to denote the one step change in the random variable $X$. As every expectation taken in this section is conditional on the first $i$ steps of the algorithm, we suppress the conditioning. That is, we simply write $\mathbb{E}[\cdot]$ instead of $\mathbb{E}\left[\cdot \mid \mathcal{F}_i\right]$ where $\mathcal{F}_0, \mathcal{F}_1, \ldots$ is the natural filtration generated by the algorithm.\\

We begin by providing intuition for the analysis to come. We will make some unjustified (and even not-quite-true) statements in this heuristic discussion, resuming full rigor in the next subsection (\ref{sec:goodevent}). Our analysis of the process will necessarily fall apart before our independent set contains very many vertices (compared to the total number of vertices $N$). The analysis falls apart for good reason: almost all the vertices will be closed by that point. The fate of an individual vertex is far more likely to be that it becomes closed than that it is chosen to be in the independent set. Thus, we are motivated to heuristically model the way vertices become closed. Of course, $v$ becomes closed when it is in an edge $e$ with $e \setminus \{v\}\subseteq I(i)$. We heuristically use a binomial random set to model $I(i)$.
Let $X_i$ be our binomial random set, where each vertex is in $X_i$, independently, with probability
\[
p=p_i\coloneqq \frac iN\eqpd
\]
This probability is chosen so that $\E[|X_i|] = i = |I(i)|$.
Of course, depending on one's purpose, using $X_i$ to model $I(i)$ may be a bad idea: in particular $X_i$ might contain edges of $\mathscr{H}$, while $I(i)$ of course does not. However, since the event that there exists an edge with {\em all but one vertex} in $X_i$ does not seem correlated with the event that an edge has {\em all vertices} in $X_i$, using this model to predict the closure of vertices is a bit more believable.  \\

Let $v$ be a fixed vertex. The expected number of witnesses to the vertex $v$ being closed, that is, the expected number of edges $e \in \mathcal{H}$ such that $v \in e$ and $e \backslash\{v\} \subseteq S_i$, is
\[
D p^{r-1}=D\left(\frac{i}{N}\right)^{r-1} \eqpd 
\]
To simulate our random greedy process by a continuous time process, we parameterize a Poisson random variable by this expectation. Thus we introduce ``scaled time'', letting $t_i^{r-1} \coloneqq D\left(\frac{i}{N}\right)^{r-1}$, so that 
\be\label{eq:t}
t_i  \coloneqq \frac{D^{\frac{1}{r-1}}\cdot i}{N} \eqpd
\ee
Going forward we let $t= t_i$ (suppressing the subscript) unless the step $i$ is noted otherwise. We will define a family of random variables that all depend on $i$, and claim that a.a.s.~ each such variable $Z=Z(i)$  stays close to an appropriate deterministic function $z=z(t)$. Justifying this claim will involve investigating how $Z$ changes over time, and for our approximation to hold we should have
\be \label{eq:changederivative}
\E[\Delta Z(i)] \approx \Delta z(t) \approx z'(t) \Delta t = z'(t) (t_{i+1}-t_i) = \frac{D^{\frac{1}{r-1}}}{N} z'(t) 
\ee

In continuizing the process, we think of the number of times a vertex $v$ is closed as being distributed according to the Poisson distribution. Thus, the probability that this quantity is zero for a fixed vertex (i.e.~the vertex is open) is given by  
\be \label{eq:q}
q=q(t)\coloneqq e^{-t^{r-1}}
\ee
which allows us to think of $q$ as the  probability a vertex is open in this continuous time setting. 

As our main goal is to prove dynamic concentration of $|V(i)|$, which is the number of vertices that remain in the hypergraph in step $i$, it will be necessary to track the following variables: For every vertex $v \in V(i)$ and $\ell=2, \ldots, r$, we define $d_{\ell}(v)=d_{\ell}(v,i)$ to be the number of edges of cardinality $\ell$ in $\mathcal{H}(i)$ that contain $v$.  
Treating $q$ as the probability that a vertex is in $V(i)$ (rather than selected to belong to the independent set) we should have $|V(i)| \approx$ $q(t) N$. Furthermore, if we heuristically assume that the distribution of the number of open vertices is binomially distributed with parameter $q$, $d_{\ell}(v) \coloneqq d_{\ell}(v,i)$ should follow the trajectory
\be\label{eq:s_ell}
s_{\ell}(t)\coloneqq D\binom{r-1}{\ell-1} q^{\ell-1} p^{r-\ell}=\binom{r-1}{\ell-1} D^{\frac{\ell-1}{r-1}} t^{r-\ell} q^{\ell-1} .
\ee

In what is to come, we separate the accumulated positive contributions to $d_{\ell}(v)$, $d^{+}_{\ell}(v)$, from the negative contributions, $d^{-}_{\ell}(v)$. We write $d_r(v)=D-d_r^{-}(v)$, and for $\ell<r$ we write $d_{\ell}(v) \coloneqq d_{\ell}(v,i )=d_{\ell}(v,0) +d_{\ell}^{+}(v,i)-d_{\ell}^{-}(v,i)$, where $d_{\ell}^{+}(v), d_{\ell}^{-}(v)$ are non-negative random variables which we use to track the number of edges of cardinality $\ell$ containing $v$ that are created and destroyed, respectively, through the first $i$ steps of the process. We define

\be \label{eq:sellpmdef}
s_{\ell}^{+}(t)\coloneqq D^{-\frac{1}{r-1}} \int_0^t \frac{\ell s_{\ell+1}(\tau)}{q(\tau)} d \tau \quad s_{\ell}^{-}(t)\coloneqq D^{-\frac{1}{r-1}} \int_0^t \frac{(\ell-1) s_{\ell}(\tau) s_2(\tau)}{q(\tau)} d \tau \eqcom
\ee
and claim that we should have $d_{\ell}^{ \pm} \approx s_{\ell}^{ \pm}$. Note that $s_{\ell}$ satisfies the differential equation
$$
s_{\ell}^{\prime}=\frac{\ell s_{\ell+1}-(\ell-1) s_{\ell} s_2}{q}=\left(s_{\ell}^{+}\right)^{\prime}-\left(s_{\ell}^{-}\right)^{\prime}
$$
and so $s_{\ell}=s_{\ell}^{+}-s_{\ell}^{-}$, (and thus we expect $d_{\ell} \approx s_{\ell}$).  
We claim that the choice of $s_{\ell}^{ \pm}$ is natural in light of the main contributions to $d_\ell$. Indeed, for each edge in $d_{\ell+1}(v)$ there are $\ell$ vertices that could be put into our independent set which would cause this edge to now be in $d_{\ell}(v)$, each one occurring with probability $1/|V(i)|$. Thus 
\[
\mathbb{E}\left[\Delta d_{\ell}^{+}\right] \approx \frac{1}{|V(i)|} \cdot \ell d_{\ell+1} \approx \frac{1}{Nq} \cdot \ell s_{\ell+1} = \frac{D^{\frac{1}{r-1}}}{N} (s_\ell^+)'\eqcom
\]
which is what we would expect (see \eqref{eq:changederivative}). For negative contributions we will see that the main way for edges in $d_{\ell}(v)$ to be destroyed is for their vertices (other than $v$) to get closed. Each such edge has $\ell-1$ vertices other than $v$ and the number of ways to close each $v'$ is $d_2(v') \approx s_2$, each occurring with probability $1/|V(i)|$. Thus
$$
\mathbb{E}\left[\Delta d_{\ell}^{-}\right] \approx \frac{1}{|V(i)|} \cdot(\ell-1) d_{\ell} d_2 \approx \frac{1}{Nq} \cdot(\ell-1) s_{\ell} s_2 = \frac{D^{\frac{1}{r-1}}}{N} (s_\ell^-)'.
$$

\subsection{The good event $\cE_i$}\label{sec:goodevent}

Here we define the {\em good event} $\cE_i$, which will stipulate that our random variables are bounded appropriately. First, we will define some additional random variables. 

\begin{definition}
    Let $M(i)$ denote the set of set of vertices in $v \in V(i)$ such that there exists a vertex $u \in I(i) \cap B(v)$.
\end{definition}

At first glance, the importance of $M(i)$ might not seem clear, however it it is precisely those vertices in $M(i)$ that we will end up having to handle separately in later calculations. If two vertices $u$ and $v$ have large $(r-1)$-codegree and one, say $v$, is chosen to be in the independent set, then by convention, every edge containing $u$ and a subset counted by the codegree is removed. Thus, we will use a more coarse bound when considering the degree of a vertex in $M(i)$ rather than carefully showing that the degree remains close to its continuous time trajectory. This ends up being the cost of being able to handle a small number of vertices with high codegree, which in turn is one of the central technical differences in the proofs of Theorem \ref{thm:ind} and Theorem \ref{thm:bennettBohman}.
\begin{definition}
    Let $d'_\ell(v, i)$ be the number of edges of size $\ell$ containing $v$ and some vertex $u \neq v$ with $u \in M(i)$. Let $d''_{\ell}(v, i)$ be the number of edges of size $\ell$ containing $v$ and a pair of vertices $u, w$ different from $v$ such that $u, w$ form a bad pair. 
\end{definition}

\begin{definition}[Degrees of Sets] \label{def:DegreesOfSetsi}For a set of vertices $A$, let $d_{A \uparrow b}(i)$ be the number of edges of size $b$ containing $A$ in $\mathscr{H}(i)$.
\end{definition}

\begin{definition}[Codegrees] \label{def:CoDegree}For a pair of distinct vertices $v, v^{\prime}$, let $c_{a, a^{\prime} \rightarrow k}\left(v, v^{\prime}, i\right)$ be the number of pairs of edges $e, e^{\prime}$, such that $v \in e \backslash e^{\prime}, v^{\prime} \in e^{\prime} \backslash e,|e|=a,\left|e^{\prime}\right|=a^{\prime}$ and $\left|e \cap e^{\prime}\right|=k$ and $e, e^{\prime} \in \mathscr{H}(i)$.
\end{definition}

Our analysis of the process will involve some constants $\zeta, \delta, \lambda$. For now we will say  
\be \label{eq:constants}
\delta:=\frac 1{10},\quad  \lambda := \frac{1}{100r},\quad 0 < \zeta \ll \lambda,
\ee
where by $\zeta \ll \lambda$ we mean that there exists an increasing function $f$ such that our argument is valid if we choose $\zeta < f(\lambda)$.
Define
\begin{align}
D_{a \uparrow b} & \coloneqq  D^{\frac{b-a}{r-1}} \phi^{1-(2r-2b) \lambda} \qquad \mbox{ for } 2 \le a < b \le r \label{eq:Darrowdef}\\
D_{1 \uparrow b} & \coloneqq  D^{\frac{b-1}{r-1}} \phi^{-(2r-2b) \lambda} \qquad \mbox{ for } 2 \le b \le r \label{eq:D1arrowdef}\\
C_{a, a^{\prime} \rightarrow k} & \coloneqq 2^r D^{\frac{a+a^{\prime}-k-2}{r-1}}\phi^{1-(4 r-2 k-2) \lambda} \qquad \mbox{ for } 2 \le a, a' < k \le r \label{eq:Carrowdef}\\
D'_{\ell} &\coloneqq D^{\frac{\ell-1}{r-1}}\phi^{1-(4r-2\ell) \lambda} \qquad \mbox{ for } 2 \le \ell \le r \label{eq:DPrime}\\
D''_{\ell} & \coloneqq  D^{\frac{\ell-1}{r-1}}\phi^{1-(2r-2\ell) \lambda} \qquad \mbox{ for } 2 \le \ell \le r \label{eq:DDoublePrime}
\end{align}

Note that the the power of $D$ in the equations above can be derived by a similar heuristic argument to \eqref{eq:s_ell}.
We will define
\be \label{eq:i_max}
i_{\max }\coloneqq \zeta N D^{-\frac{1}{r-1}} \log ^{\frac{1}{r-1}} (1/\phi) \quad \mbox{ and } \quad t_{\max }\coloneqq \frac{D^{\frac{1}{r-1}}}{N} \cdot i_{\max }=\zeta \log ^{\frac{1}{r-1}} (1/\phi). 
\ee

Note that
\be\label{eq:qtMax}
q\left(t_{\max }\right)=\phi^{\zeta^{r-1}} .
\ee

Let us define the \textit{good event} $\mathcal{E}_i$ to be the event that, in the $i$th step, the following statements all hold. Below, we use some deterministic functions $f_V$ and $f_2, \ldots, f_r$, which we call {\em error functions}. It is important that these functions satisfy certain inequalities which we will discuss in Appendix \ref{sec:errorfunctions} (see Lemma \ref{lem:deterministicbounds}). 
\be\label{eq:DynamicVBound}
|V(i)| \in N q \pm N \phi^\delta f_V\eqcom
\ee
\be\label{eq:DynamicDegBound}
d_{\ell}^{ \pm}(v, i) \in s_{\ell}^{ \pm} \pm D^{\frac{\ell-1}{r-1}} \phi^\delta f_{\ell} \quad \text { for } \ell=2, \ldots, r \text { and all } v \in V(i) \setminus M(i) \eqcom
\ee
\be \label{eq:DynamicSetDegBound}
d_{A \uparrow b}(i) \leq D_{a \uparrow b} \quad \text { for } 1 \leq a<b \leq r \text { and all } A \in\binom{V(i)}{a} \eqcom
\ee
\be\label{eq:DynamicCoDegBound}
c_{a, a^{\prime} \rightarrow k}\left(v, v^{\prime}, i \right) \leq C_{a, a^{\prime} \rightarrow k} \quad \text { for all non-bad pairs } v, v^{\prime} \in V(i)
\ee
\be\label{eq:Dynamicdprime}
d'_\ell(v, i) \leq D'_\ell \quad \text { for all } v \in V(i)\eqcom
\ee
\begin{center}
and
\end{center}
\be\label{eq:Dynamicddubprime}
d''_\ell(v, i) \leq D''_\ell \quad \text { for all } v\in V(i)\eqpd
\ee
.
In addition, let the stopping time $\tau$ be the minimum of $i_{\max}$ as defined above, and the first step, $i$, such that one of the conditions of the good event is violated.\\

We now state our main technical theorem, which may have further applications.
\begin{theorem}\label{thm:tech}
    $\cE_{i_{\max}}$ holds asymptotically almost surely.
\end{theorem}
Theorem \ref{thm:ind} follows from Theorem \ref{thm:tech}. Indeed, if $\cE_{i_{\max}}$ holds then by \eqref{eq:DynamicVBound} we have $V(i_{\max}) \ge Nq(t_{\max}) - N\phi^\delta f_V(t_{\max}),$ which is positive by Lemma \ref{lem:deterministicbounds} (7). Thus, there are open vertices remaining at step $i_{\max}$ and in particular the process lasts at least that many steps. We spend the next two sections proving Theorem \ref{thm:tech}. The proof proceeds by bounding the probability that $\cE_{i_{\max}}$ fails due to any of the possible conditions \eqref{eq:DynamicVBound}--\eqref{eq:Dynamicddubprime}. 

\section{Crude Bounds} In this section (and Appendix \ref{sec:deferred}) we bound the probability that $\cE_i$ occurs for $i<i_{\max}$ due to the failure of one of the Conditions \eqref{eq:DynamicSetDegBound}--\eqref{eq:Dynamicddubprime}. Here we handle \eqref{eq:DynamicSetDegBound} and \eqref{eq:Dynamicdprime}, and the rest are in Appendix \ref{sec:deferred} since the arguments are repetitive.\\

Using   \eqref{eq:DynamicVBound} and Lemma \ref{lem:deterministicbounds} part (8),  we have the following bound on $|V(i)|$ for $i \le i_{\max}$:

\be\label{eq:Vlowerbound}
\begin{split}
|V(i)| &\geq N\rbrac{q-\phi^{\delta}f_V} \ge N\phi^{\lambda}\eqpd
\end{split}
\ee

We begin by showing that condition \eqref{eq:DynamicSetDegBound} holds until step $i_{\max}$ a.a.s. First we handle the case where $a \ge 2$. 
\begin{lemma} Let $2 \leq a<b \leq r$, and recall that $D_{a \uparrow b}  \coloneqq D^{\frac{b-a}{r-1}} \phi^{1-(2r-2b) \lambda}.$ Then,
\label{lem:degreeBound}
\be\label{eq:setdegreeprobabilitybound}
\mathbb{P}\left(\exists i \leq \tau \text { and } A \in\binom{V(i)}{a} \text { such that } d_{A \uparrow b}(i) \geq D_{a \uparrow b}\right) =o(1) \eqpd
\ee
\end{lemma}

\begin{proof}

 We define variables $Y_{A \uparrow b}(i)$ as follows.

$$
Y_{A \uparrow b}(i)\coloneqq  \begin{cases}d_{A \uparrow b}(i)-\frac{r}{N}  D^{\frac{b-a+1}{r-1}} \phi^{1-(2 r-2 b-1) \lambda }\cdot i & \text { if } \mathcal{E}_{i-1} \text { holds } \\ Y_{A \uparrow b}(i-1) & \text { otherwise }\end{cases}
$$

We will bound $d_{A \uparrow b}(i)$ by appealing to the following lemma due to Freedman \cite{freedman}:

\begin{lemma}[Freedman]\label{lem:FreedmanLemma} Let $Y(i)$ be a supermartingale, with $\Delta Y(i) \leq C$ for all $i$, and 
\[
W(i)\coloneqq \sum_{k \leq i} \operatorname{Var}\left[\Delta Y(k) \mid \mathcal{F}_k\right].
\]
Then
$$
\mathbb{P}[\exists i: W(i) \leq w, Y(i)-Y(0) \geq d] \leq \exp \left(-\frac{d^2}{2(w+C d)}\right)\eqpd
$$
\end{lemma}

To show that $Y_{A \uparrow b}(i)$ is a supermartingale, it suffices to observe  
\[
\mathbb{E}\left[\Delta d_{A \uparrow b}(i) \right] = \frac{(b+1-a)d_{A \uparrow b+1}}{|V(i)|}\le \frac{r}{N}  D^{\frac{b-a+1}{r-1}} \phi^{1-(2 r-2 b-1) \lambda }.
\]

 For our application of Freedman's theorem \eqref{lem:FreedmanLemma}, we can set $C  =D_{a+1 \uparrow b+1} =D^{\frac{b-a}{r-1}}\phi^{1-(2 r-2 b-2) \lambda}$. We remind the reader that all expectations are conditioned on the history of the process (even though we omit this conditioning in our notation), and furthermore we will start using this convention on variance as well. With that in mind, we have
$$
\begin{aligned}
\mathbb{V} a r\left[\Delta Y_{A \uparrow b}(i)\right]=\mathbb{V} a r\left[\Delta d_{A \uparrow b}(i)\right] & \leq \mathbb{E}\left[(\Delta d_{A \uparrow b}(i))^2\right] \\
& \leq C\cdot \mathbb{E}\left[\left|\Delta d_{A \uparrow b}(i)\right| \right] \\
& \leq \frac{r}{N} D^{\frac{2 b-2 a+1}{r-1}} \phi^{2-(4 r-4 b-3) \lambda}.
\end{aligned}
$$

Using \eqref{eq:i_max} we have $W({i_{\max}}) \leq i_{\max} \cdot \frac{r}{N} D^{\frac{2 b-2 a+1}{r-1}} \phi^{2-(4 r-4 b-3) \lambda} \le D^{\frac{2 b-2 a}{r-1}} \phi^{2-(4 r-4 b-3) \lambda} \log(1/\phi),$ so we can take $w$ to be the latter expression. Setting $d=D^{\frac{b-a}{r-1}}\phi^{1-(2r-2b-1)\lambda} $ and noting that $Y_{A \uparrow b}(0)\le \Delta_a(\mathscr{H}^{(b)}) \le \phi D^{\frac{b-a}{r-1}}$ by Condition \eqref{cond:maxLdegree},  Freedman's theorem gives us that 

\begin{align}
\begin{split}
P\left[\exists i \le i_{\max} : Y_{A \uparrow b}(i) - \phi D^{\frac{b-a}{r-1}} \geq d \right]& \le \exp \left(-\frac{d^2}{2(w+C d)}\right)\\
&\le \exp \left(-\frac{D^{\frac{2(b-a)}{r-1}}\phi^{2-(4r-4b-2)\lambda}}{2\left(D^{\frac{2 b-2 a}{r-1}} \phi^{2-(4 r-4 b-3) \lambda} \log(1/\phi)+D^{\frac{2(b-a)}{r-1}}\phi^{2-(4r-4b-3)\lambda}\right)}\right) \\
& < \exp \rbrac{-\phi^{-\lambda/2}} = o(N^{-r}),
\end{split}
\end{align}
where on the last line we used that $\lambda = 1/100r$ and $\phi < \log^{-300r}N$. The probability above is small enough to beat a union bound over all $O(N^r)$ choices for the set $A$, and so a.a.s.~the above event does not happen for any $A$. In other words, $Y_{A \uparrow b}(i)-\phi D^{\frac{b-a}{r-1}} <d$ for all $i \le i_{\max}$. Thus we have 
\[
d_{A \uparrow b}(i) \le \phi D^{\frac{b-a}{r-1}}+ d + \frac{r}{N}  D^{\frac{b-a+1}{r-1}} \phi^{1-(2 r-2 b-1) \lambda } \cdot i_{\max} < D_{a \uparrow b}.
\]
\end{proof}
We now bound the probability that $\mathcal{E}_{i_{\text {max }}}$ fails due to condition \eqref{eq:Dynamicdprime}.

\begin{lemma} 
Recall that $D'_{\ell} \coloneqq  D^{\frac{\ell-1}{r-1}}\phi^{1-(4r-2\ell) \lambda}$. Then,
\be\label{eq:dprimeprobabilitybound}
\mathbb{P}\left(\exists i \leq \tau \text { and } v \in V(i) \text { such that } d'_{\ell}(v, i) \geq D'_{\ell}\right) =o(1) \eqpd
\ee
\end{lemma}

\begin{proof}

 We define variables $Y_\ell(v)$ as follows.

$$
Y_\ell(v, i)\coloneqq  \begin{cases}d'_\ell(v, i)-\frac{3r D^{\frac{\ell}{r-1}}\phi^{1-(4r-2\ell-1)\lambda}}{N }\cdot i & \text { if } \mathcal{E}_{i-1} \text { holds } \\ Y_\ell(v, i-1) & \text { otherwise }\end{cases}
$$

To show that $Y_\ell(v)$ is a supermartingale, we bound $\mathbb{E}\left[\Delta d'_\ell(v,i) \right]$. There are three possible types of positive contribution to $\Delta d'_\ell(v,i)$. The expected contribution from edges counted by $d'_{\ell+1}(v, i-1)$ having one of their vertices  chosen to be in our independent set is at most  
\be \label{eq:deltad'1}
\frac{\ell d'_{\ell+1}(v, i-1) }{ |V(i)|} \le \frac{r D^{\frac{\ell}{r-1}}\phi^{1-(4r-2\ell-2)\lambda}}{N \phi^\lambda} = \frac{r D^{\frac{\ell}{r-1}}\phi^{1-(4r-2\ell-1)\lambda}}{N }
\ee
The expected contribution from edges 
counted by $d_\ell(v)$ containing some vertex $u$ such that we choose $w \in B(u)$ in step $i$ is at most (recalling $|B(u)| \le D^{\frac 1{r-1}}\phi$)
\be \label{eq:deltad'2}
\frac{(\ell-1)d_{\ell}(v, i)  D^{\frac{1}{r-1}} \phi}{|V(i)|} \le \frac{r \left(s_\ell + 2D^\frac{\ell-1}{r-1}\phi^\delta f_\ell \right) D^{\frac{1}{r-1}} \phi}{N \phi^\lambda} \le \frac{r \phi^{1-\lambda}D^{\frac{\ell}{r-1}}\left(C_1+C_2\phi^{\delta-\lambda} \right)}{N}.
\ee
Here, the rightmost inequality follows, with the appropriate constants $C_1$ and $C_2$, from Lemma \ref{lem:deterministicbounds} parts (10) and (11). Letting $C^{\ast} \coloneqq r\cdot \max\{C_1, C_2\}$, we have the above is at most 

 \be
    \frac{C^{\ast} \cdot 
    \phi^{1-2\lambda}D^{\frac{\ell}{r-1}}}{N}= o\rbrac{ \frac{ D^{\frac{\ell}{r-1}}\phi^{1-(4r-2\ell-1)\lambda}}{N }}\eqcom 
    \ee
    where the last equality holds since we assume $r \geq 3$.

Finally, we bound the expected contribution from edges in $d''_{\ell+1}(v)$ containing two vertices $u,u'$ such that $u' \in B(u)$  and we choose $u'$ in step $i$. Thus the expected contribution from this case is at most 
\be \label{eq:deltad'3}
\frac{\ell  D''_{\ell+1}}{|V(i)|} \le  \frac{r D^{\frac{\ell}{r-1}}\phi^{1-(2r-2\ell-2)\lambda}}{N \phi^\lambda} = \frac{r D^{\frac{\ell}{r-1}}\phi^{1-(2r-2\ell-1)\lambda}}{N }= o
\left(\frac{D^{\frac{\ell}{r-1}}\phi^{1-(4r-2\ell-1)\lambda}}{N } \right).
\ee

Summing the three contributions \eqref{eq:deltad'1}, \eqref{eq:deltad'2}, \eqref{eq:deltad'3} we have 
\be \label{eq:deltad'}
\mathbb{E}\left[\Delta\left|d'_\ell(v)(i)\right| \right] \le \frac{3r D^{\frac{\ell}{r-1}}\phi^{1-(4r-2\ell-1)\lambda}}{N }.
\ee

Thus, $\mathbb{E}\left[\Delta\left|Y_\ell(v)(i)\right|\right] \leq 0$. For our application of Freedman's  \ref{lem:FreedmanLemma}, we will set $C  =2D^{\frac{\ell-1}{r-1}} \phi^{1-(2r-2\ell-2) \lambda}$, which we justify as follows. The contribution to $ \Delta d'_\ell(v)$ from edges of size $\ell+1$ is at most $D_{2 \uparrow \ell+1}=D^{\frac{\ell-1}{r-1}} \phi^{1-(2r-2\ell-2) \lambda}$, and the contribution from edges of size $\ell$ (which in this step go from $d_\ell(v)$ to $d'_\ell(v)$) is at most $D^{\frac{1}{r-1}} \phi D_{2 \uparrow \ell}=D^{\frac{\ell-1}{r-1}} \phi^{2-(2r-2\ell) \lambda}$. Thus our choice of $C$ is justified. Note that we have

$$
\begin{aligned}
\mathbb{V} a r\left[\Delta Y_\ell(v)\right]=\mathbb{V} a r\left[\Delta d'_\ell(v)\right] & \leq \mathbb{E}\left[(\Delta d'_\ell(v))^2\right] \\
& \leq C\cdot \mathbb{E}\left[\left|\Delta d'_\ell(v)\right| \right] \\
& \leq \frac{6r D^{\frac{2\ell-1}{r-1}}\phi^{2-(6r-4\ell-3)\lambda}}{N }
\end{aligned}
$$
where on the last line we used \eqref{eq:deltad'}. Thus we have 
\[
W({i_{\max}}) \leq i_{\max} \cdot\frac{6r D^{\frac{2\ell-1}{r-1}}\phi^{2-(6r-4\ell-3)\lambda}}{N }=O\left(D^{\frac{2\ell-2}{r-1}}\phi^{2-(6r-4\ell-3)\lambda} \log(1/\phi)\right)
\]
so we can take $w$ to be the latter expression. Setting $d=\frac12 D^{\frac{\ell-1}{r-1}} \phi^{1-(4r-2\ell)\lambda}$ and observing that $Y_\ell(v, 0)=0$, Freedman's theorem gives us that 

\begin{align}
\begin{split}
P\left[\exists i \le i_{\max} : Y_\ell(v, i) \geq d \right]& \le \exp \left(-\frac{d^2}{2(w+C d)}\right)\\
&=\exp \left( -\Omega \left(\frac{D^{\frac{2\ell-2}{r-1}} \phi^{2-(8r-4\ell)\lambda}}{D^{\frac{2\ell-2}{r-1}}\phi^{2-(6r-4\ell-3)\lambda} \log(1/\phi) + D^{\frac{\ell-1}{r-1}} \phi^{1-(2r-2\ell-2) \lambda} \cdot D^{\frac{\ell-1}{r-1}} \phi^{1-(4r-2\ell)\lambda}} \right)\right)\\
&=\exp \left( -\Omega \left(\frac{1}{\left(\phi^{(2r+3)\lambda}\log(1/\phi)+\phi^{(2r+2)\lambda}\right)} \right)\right)\\
&\le \exp \left(-\Omega\rbrac{\phi^{-(2r+2)\lambda}}\right) =o(N^{-1}).
\end{split}
\end{align}
If the unlikely event above does not happen, i.e. if $Y_\ell(v, i) <d$ for all $i \le i_{\max}$, then 
\[
d'_\ell(v, i) < d + \frac{3r D^{\frac{\ell}{r-1}}\phi^{1-(4r-2\ell-1)\lambda}}{N } \cdot i_{\max} < D'_{\ell}\eqcom
\]
where the final inequality holds by our choice of $\zeta$.
 Now \eqref{eq:dprimeprobabilitybound} follows from the union bound over at most $N$ choices for $v$.
\end{proof}

\section{Dynamic Concentration}
In this section we bound the probability that $\cE_i$ occurs for $i<i_{\max}$ due to the failure of one of the conditions \eqref{eq:DynamicVBound} or \eqref{eq:DynamicDegBound}. We begin by defining a set of random variables whose trajectories will need to be kept track of throughout the duration of the process. We'll go on then to prove that these variables are concentrated around their mean, giving us the ability to bound the length of the process of random greedy independent set process.
Consider the sequences
\be
\begin{aligned}\label{eq:DefinitionOfZ}
Z_V=Z_V(i) & \coloneqq  \begin{cases}|V(i)|-N q-N \phi^{\delta} f_V & \text { if } \mathcal{E}_{i-1} \text { holds } \\ Z_V(i-1) & \text { otherwise }\end{cases}& & \\
Z_{\ell}^{+}(v)=Z_{\ell}^{+}(v, i) & \coloneqq  \begin{cases}d_{\ell}^{+}(v)-s_{\ell}^{+}-D^{\frac{\ell-1}{r-1}}\phi^{\delta} f_{\ell} & \text { if } \mathcal{E}_{i-1} \text { holds and } v\notin M(i)\cup I(i) \\ Z_{\ell}^{+}(v, i-1) & \text { otherwise }\end{cases}& & \text { for } 2 \leq \ell \leq r-1 \\
Z_{\ell}^{-}(v)=Z_{\ell}^{-}(v, i) & \coloneqq \begin{cases}d_{\ell}^{-}(v)-s_{\ell}^{-}-D^{\frac{\ell-1}{r-1}}\phi^{\delta} f_{\ell}  & \text { if } \mathcal{E}_{i-1} \text { holds and } v\notin M(i)\cup I(i) \\ Z_{\ell}^{-}(v, i-1) & \text { otherwise }\end{cases}& & \text { for } 2 \leq \ell \leq r\\
\end{aligned}
\ee
Where we can think of $Z_V$ as the deviation of the size of the set of open vertices from its continuous time mean. Similarly, we think of $Z^{+}_{\ell}$ as being the deviation of the \textit{positive contribution} to the degree of a subset of size $\ell$ containing a fixed vertex $v$, from its mean, with $Z^{-}_{\ell}$ defined analogously.\\

In this section, we establish the upper bound on $V(i)$ in \eqref{eq:DynamicVBound} by showing that $Z_V<0$ for all $i \leq \tau$ a.a.s. Similarly, we establish the upper bounds on $d_{\ell}^{ \pm}(v)$ in \eqref{eq:DynamicDegBound} by showing that $Z_{\ell}^{ \pm}(v)<0$ for all $i \leq \tau$ a.a.s. We begin by showing that the sequences $Z_V$ and $Z_{\ell}^{ \pm}$ are supermartingales conditionally on the collection of error functions $\left\{f_V\right\} \cup\left\{f_{\ell} \mid \ell=2, \ldots, r\right\}$ having an ``appropriate'' value. 
We will see that each of our calculations imposes a condition (inequality) on the collection of error functions and their derivatives. These differential equations are the \textit{variation equations}, and we can choose error functions that satisfy the variation equations after completing the expected change calculations.
The functions will be chosen so that all error functions evaluate to 1 at $t=0$ and are increasing in $t$.\\

After choosing the error functions appropriately such that our random variable sequences defined in \eqref{eq:DefinitionOfZ} are indeed supermartingales, we use the fact that they have initial values that are negative and relatively large in magnitude. We complete the proof by applying martingale deviation inequalities to show that it is very unlikely for these supermartingales to ever be positive. Note that, to obtain the necessary lower bounds on $V(i)$ and $d_{\ell}^{ \pm}(v)$ respectively, we can apply an analogous argument to the sequences

\begin{align*}
\Zs_V=\Zs_V(i) & \coloneqq  \begin{cases}|V(i)|-N q+N \phi^{\delta} f_V & \text { if } \mathcal{E}_{i-1} \text { holds } \\ \Zs_V(i-1) & \text { otherwise }\end{cases}& & \\
\Zs_{\ell}^{+}(v)=\Zs_{\ell}^{+}(v, i) & \coloneqq  \begin{cases}d_{\ell}^{+}(v)-s_{\ell}^{+}+D^{\frac{\ell-1}{r-1}}\phi^{\delta} f_{\ell} & \text { if } \mathcal{E}_{i-1} \text { holds and } v\notin M(i)\cup I(i) \\ \Zs_{\ell}^{+}(v, i-1) & \text { otherwise }\end{cases}& & \text { for } 2 \leq \ell \leq r-1 \\
\Zs_{\ell}^{-}(v)=\Zs_{\ell}^{-}(v, i) & \coloneqq \begin{cases}d_{\ell}^{-}(v)-s_{\ell}^{-}+D^{\frac{\ell-1}{r-1}}\phi^{\delta} f_{\ell}  & \text { if } \mathcal{E}_{i-1} \text { holds and } v\notin M(i)\cup I(i) \\ \Zs_{\ell}^{-}(v, i-1) & \text { otherwise }\end{cases}& & \text { for } 2 \leq \ell \leq r\\
\end{align*}

and show that they are also, conditionally on the appropriate choice of error functions, submartingales. As the calculations are extremely similar, and show that the same variation equations suffice in both cases, we omit them (although some of the work we will show for the supermartingales in \eqref{eq:DefinitionOfZ} can be used as is for the submartingales).\\

We will now start showing that  $Z_V$ is a supermartingale, i.e. that $\E[\Delta Z_V(i)] \le 0$. Recall that we are always conditioning on the history of the process up to step $i$. Since we have $\Delta Z_V(i)=0$ if the event $\mc{E}_i$ does not hold (see \eqref{eq:DefinitionOfZ}), we can restrict our attention to the case that $\mc{E}_i$ holds. So we are conditioning on the history up to step $i$ and assuming that this history is such that $\mc{E}_i$ holds. We treat this first case in some detail in an effort to illuminate our methods for the reader who is not familiar with these techniques. Let $S\coloneqq N D^{-\frac{1}{r-1}}$ and recall (from \ref{eq:t}) that $t=i / S$. (The quantity $S$ is sometimes called the time scaling.) We write
\[
\Delta Z_V= \Delta |V(i)| -N\big(q\left(t+1 / S\big)-q(t)\right)-N \phi^{\delta}\big(f_V\left(t+1 / S\right)-f_V(t)\big),
\]
and make use of the estimates
\[
\begin{gathered}\label{eq:taylor}
q\left(t+1 / S\right)-q(t)=\frac{q^{\prime}(t)}{S}+O\left(\frac{q^{\prime \prime}}{S^2}\right) \\
f_V\left(t+1 / S\right)-f_V(t)=\frac{f_V^{\prime}(t)}{S}+O\left(\frac{f_V^{\prime \prime}}{S^2}\right)
\end{gathered}
\]
where, in the big-O terms, by $q^{\prime \prime}$ and $f_V^{\prime \prime}$ we mean bounds on these second derivatives that hold uniformly in the interval of interest, i.e.~$t \in [0, t_{max}]$. Thus, using $q^{\prime}=-s_2 D^{-\frac{1}{r-1}}$ we obtain
\begin{align}
    \Delta Z_V&= \Delta |V(i)| + s_2 - D^\frac{1}{r-1} \phi^\delta f_V^{\prime}+O\left(D^\frac{2}{r-1} N^{-1} (q^{\prime \prime}+ \phi^\delta f_V^{\prime \prime})\right) \nonumber\\
    &= \Delta |V(i)| + s_2 - D^\frac{1}{r-1} \phi^\delta f_V^{\prime}+O\left(D^\frac{2}{r-1} N^{-1} \phi^{-\lambda}\right), \label{eq:deltaZV}
\end{align}

When we take expectations of both sides in \eqref{eq:deltaZV}, we will see that the main terms of $\E[\Delta|V(i)|]$ and $N\left(q\left(t+1 / S_t\right)-q(t)\right)$ cancel exactly. The second order terms from $\E[\Delta|V(i)|]$ will be then be balanced by the $N \phi^{\delta}\left(f_V\left(t+1 / S_t\right)-f_V(t)\right)$ term. We now proceed with the explicit details.

\subsection{Showing $Z_V$ is a supermartingale}

We estimate $\mathbb{E}[\Delta |V(i)|]$.  We have (explanation follows)
\begin{align}
\mathbb{E}\left[\Delta |V(i)|\right] & =
-1-\frac{1}{|V(i)|}\sum_{v \in V(i)} d_2(v)  \nn \\
&\subseteq -1-\frac{1}{|V(i)|}\left[|V(i) \setminus M(i)|\left(s_2 \pm 2D^\frac{1}{r-1} \phi^\delta f_2\right) +O(| M(i)| \cdot D_{1 \uparrow 2})\right] \nn\\
 &\subseteq -\frac{|V(i)|+O(N \phi^{1-\lambda})}{|V(i)|}\left(s_2 \pm 2D^\frac{1}{r-1} \phi^\delta f_2\right)+O\left(D^{\frac{1}{r-1}} \phi^{1-(2r-2) \lambda} \right) \nn\\
 &\subseteq- s_2 \pm 2D^\frac{1}{r-1} \phi^\delta f_2 + O\left(D^{\frac{1}{r-1}} \phi^{1-(2r-2) \lambda}\right).\label{eq:deltaVbound}
\end{align}
Notice that we treat differently the vertices in $v \in M(i)$, only using rough estimates on the degree of subsets of size two containing those vertices. Indeed, on the second line we have used \eqref{eq:DynamicDegBound}  and \eqref{eq:DynamicSetDegBound}. On the third line we have used \eqref{eq:Vlowerbound} and the following: since $|B(v)| \le D^\frac{1}{r-1} \phi$ for all $v$, we have
\begin{equation}
    \label{eq:Mbound}
|M(i)| \le i_{\max} \cdot D^\frac{1}{r-1} \phi <N \phi^{1-\lambda}.
\end{equation}
Finally, to get \eqref{eq:deltaVbound} we use \eqref{eq:Vlowerbound} and Lemma \ref{lem:deterministicbounds} parts (10) and (11) which imply 
\[
 s_2 + 2D^\frac{1}{r-1} \phi^\delta f_2 = O\left( D^\frac{1}{r-1} \phi^{-\lambda}\right).
\]

This together with the fact that $r\geq 3$ yields
\[
\frac{O\left(N\phi^{1-\lambda}\right)}{|V(i)|}\left(s_2 \pm 2D^\frac{1}{r-1} \phi^\delta f_2\right) = O\left( D^\frac{1}{r-1} \phi^{1-3\lambda}\right) = O\left(D^{\frac{1}{r-1}} \phi^{1-(2r-2) \lambda}\right),
\]
justifying the big-O term in \eqref{eq:deltaVbound}.\\

Now, taking the expectation of \eqref{eq:deltaZV} and using \eqref{eq:deltaVbound}, we obtain 
\be \label{eq:variationZV}
\mathbb{E}\left[\Delta Z_V\right]\le D^\frac{1}{r-1} \phi^\delta (2f_2-f_V') + O\left(D^{\frac{1}{r-1}} \phi^{1-(2r-2) \lambda}\right),
\ee
since the big-O term in \eqref{eq:deltaVbound} dominates the big-O term in \eqref{eq:deltaZV} by \eqref{eq:LowerBoundonN}. Now the right hand side of \eqref{eq:variationZV} is negative since $\phi^\delta$ is much larger than $\phi^{1-(2r-2)\lambda}$ and $2f_2 - f_V' = -\Omega(1)$ by Lemma \ref{lem:deterministicbounds} part (1). Therefore $Z_V$ is a supermartingale. \\

To check that $\Zs_V$ is a submartingale, one can run through the derivation of \eqref{eq:variationZV} mutatis mutandis, using $\Zs_V$ instead of $Z_V$ and lower bounds instead of upper bounds, to obtain
\[
\mathbb{E}\left[\Delta \Zs_V\right]\ge -D^\frac{1}{r-1} \phi^\delta (2f_2-f_V) + O\left(D^{\frac{1}{r-1}} \phi^{1-(2r-2) \lambda}\right) \ge 0.
\]

\subsection{Showing $Z_{\ell}^{+}$ is a supermartingale} In several ways this will be similar to what we just did for $Z_V$. In particular, in the definition \eqref{eq:DefinitionOfZ}, we have a case where $\Delta Z_\ell^+ = 0$ and so we can restrict our attention to the other case. Going forward it will be important to recall that $\Delta Z_{\ell}^{+} \coloneqq \Delta Z_{\ell}^{+}(i) = Z_{\ell}^{+}(i+1) - Z_{\ell}^{+}(i)$, i.e., we are considering the change in  $Z_{\ell}^{+}$ from step $i$ to $i+1$. In doing so, we assume we are in the good event $\mc{E}_i$, and when we estimate $\E[\Delta d_\ell^+]$ we ignore any possible contribution coming from the possibility that the $i^{th}$ vertex to be chosen, $v_i$, is such that $v_i \in B(v)$ (since in that case we would have $v \in M(i+1)$ and thus $\Delta Z_\ell^+ = 0$). We have 
\be \label{eq:DZL1}
\mathbb{E}\left[\Delta Z_{\ell}^{+}(v)\right] = \frac{1}{|V(i)|} \sum_{u \in V(i)} x_u
\ee
where $x_u$ is the change in $Z_{\ell}^{+}(v)$ if $v_i=u$. By definition \eqref{eq:DefinitionOfZ}, $x_u=0$ if $u \in B(v) \cup \{v\}$ since then $v \in M(i+1) \cup I(i+1)$. We write (explanation follows) 
\begin{align}
 \mathbb{E}\left[\Delta Z_{\ell}^{+}(v)\right] =& \frac{1}{|V(i)|} \sum_{u \in V(i) \setminus B(v) \setminus \{v\}}\rbrac{ d_{\{v, u\} \uparrow \ell+1}(i) -   \frac{D^\frac{1}{r-1}}{N} (s_\ell^+)'  - \frac{D^{\frac{\ell}{r-1}} \phi^{\delta}}{N} f_{\ell}^{\prime} +O\left(D^\frac{2}{r-1} N^{-2} ((s_\ell^+)^{\prime \prime}+ D^\frac{\ell-1}{r-1}\phi^\delta f_\ell^{\prime \prime})\right)} \nn\\
=& \frac{\ell d_{\ell+1}(v)}{|V(i)|} -  \frac{D^\frac{1}{r-1}}{N} (s_\ell^+)'  - \frac{D^{\frac{\ell}{r-1}} \phi^{\delta}}{N} f_{\ell}^{\prime} +O\left( D^\frac{\ell}{r-1}N^{-1} \phi^{2-(2r-2\ell-1)\lambda}+ D^\frac{\ell+1}{r-1}N^{-2} \phi^{-\lambda} \right)\label{eq:ZLline1}\\
\le & \frac{\ell (s_{\ell+1} + 2D^{\frac{\ell}{r-1}}\phi^{\delta} f_{\ell+1})}{Nq - N\phi^\delta f_V} -  \frac{\ell s_{\ell+1}}{Nq}  - \frac{D^{\frac{\ell}{r-1}} \phi^{\delta}}{N} f_{\ell}^{\prime} +O\left( D^\frac{\ell}{r-1}N^{-1} \phi^{1-\lambda}\right)\label{eq:ZLline2}\\
\le & \frac{D^{\frac{\ell}{r-1}} \phi^\delta}{N} \cdot\left(2 \ell q^{-1} f_{\ell+1}+\ell\binom{r-1}{\ell} t^{r-\ell-1} q^{\ell-2} f_V-f_{\ell}^{\prime}\right) +O\rbrac{D^{\frac{\ell}{r-1}}N^{-1} \phi^{2\delta - 5 \lambda}}\label{eq:ZLline3}.
\end{align}
To justify line \eqref{eq:ZLline1}, first we estimate 
\begin{align*}
 \frac{1}{|V(i)|} \sum_{u \in V(i) \setminus B(v) \setminus \{v\}} d_{\{v, u\} \uparrow \ell+1}(i) &= \frac{1}{|V(i)|}\sum_{u \in V(i) \setminus  \{v\}} d_{\{v, u\} \uparrow \ell+1}(i) + O\left(\frac{|B(v)|D_{2 \uparrow \ell+1}}{|V(i)|} \right)\\
 &= \frac{\ell d_{\ell+1}(v)}{|V(i)|} + O\left( D^\frac{\ell}{r-1}N^{-1} \phi^{2-(2r-2\ell-1)\lambda} \right).
\end{align*}

To finish justifying \eqref{eq:ZLline1}, we use Lemma \ref{lem:deterministicbounds} to bound the error terms:
\be\label{eq:slprimebound}
\frac{|B(v) \cup \{v\}|}{|V(i)|} \left( \frac{D^\frac{1}{r-1}}{N} (s_\ell^+)'  + \frac{D^{\frac{\ell}{r-1}} \phi^{\delta}}{N} f_{\ell}^{\prime} \right) +O\left(D^\frac{2}{r-1} N^{-2} ((s_\ell^+)^{\prime \prime}+ D^\frac{\ell-1}{r-1}\phi^\delta f_\ell^{\prime \prime})\right)= O\left( D^\frac{\ell+1}{r-1}N^{-2} \phi^{-\lambda} \right).
\ee
To justify line \eqref{eq:ZLline2}, we use our bounds on $d_{\ell+1}(v), |V(i)|$ and recall by \eqref{eq:sellpmdef} that $(s_{\ell}^+)' = D^{-\frac{1}{r-1}} \ell s_{\ell+1} / q$. The big-O term on line \eqref{eq:ZLline2} is just simplifying the previous big-O term using \eqref{eq:LowerBoundonN}. To justify line \eqref{eq:ZLline3} we will use the following lemma:
\begin{lemma} \label{lem:est}
For any real numbers $x, y, \e_x, \e_y$, if we have $x,y \neq 0$ and $\left|\frac{\e_x}{x}\right|, \left|\frac{\e_y}{y}\right|\le \frac{1}{2}$, then

$$\frac{x+ \e_x}{y+\e_y} - \frac{x}{y} = \frac{ \e_x}{y} - \frac{x \e_y}{y^2} + O\left(\frac{y \e_x \e_y + x \e_y^2}{y^3} \right)$$
\end {lemma}
With this lemma, we turn to \eqref{eq:ZLline2} and see that 
\begin{align*}
\frac{\ell (s_{\ell+1} + 2D^{\frac{\ell}{r-1}}\phi^{\delta} f_{\ell+1})}{Nq - N\phi^\delta f_V} -  \frac{\ell s_{\ell+1}}{Nq} &= \frac{2\ell D^{\frac{\ell}{r-1}}\phi^{\delta} f_{\ell+1}}{Nq} + \frac{\ell s_{\ell+1} N\phi^\delta f_V}{N^2q^2} + O\rbrac{\frac{N^2q D^{\frac{\ell}{r-1}}\phi^{\delta} f_{\ell+1}\phi^\delta f_V + s_{\ell+1}N^2\phi^{2\delta} f_V^2}{N^3q^3}}\\
& = \frac{2\ell D^{\frac{\ell}{r-1}}\phi^{\delta} f_{\ell+1}}{Nq} + \frac{\ell s_{\ell+1} N\phi^\delta f_V}{N^2q^2} + O\rbrac{D^{\frac{\ell}{r-1}}N^{-1} \phi^{2\delta - 5 \lambda}}
\end{align*}
where on the second line we have used Lemma \ref{lem:deterministicbounds}. This justifies \eqref{eq:ZLline3}. 

Now observe that \eqref{eq:ZLline3} is negative since $\phi^\delta$ is much larger than $\phi^{2\delta - 5 \lambda}$ and \[2 \ell q^{-1} f_{\ell+1}+\ell\binom{r-1}{\ell} t^{r-\ell-1} q^{\ell-2} f_V-f_{\ell}^{\prime} = -\Omega(1)\]
by Lemma \ref{lem:deterministicbounds} parts (2) and (3). Therefore $Z_\ell^+$ is a supermartingale. Similarly, $\Zs_\ell^+$ is a submartingale.

\subsection{Showing $Z_{\ell}^{-}$ is a supermartingale}

 The main term in the expected change of $d_{\ell}^{-}(v)$ comes from the selection of vertices $y$ for which there exists a vertex $x$ such that $\{y, x\} \in \mathscr{H}(i)$, and there is an edge $e$ counted by $d_{\ell}(v)$ such that $x \in e$. However, here we must also account for the convention of removing redundant edges. For a fixed edge $e$ counted by $d_{\ell}(v)$, the selection of any vertex in the following set results in the removal of $e$ from this count:
$$
\{y \in V(i): \exists A \subset e \text { such that } A \neq\{v\} \text { but } A \cup\{y\} \in \mathscr{H}(i)\}
$$

(Note that this is essentially restating the convention that we remove edges that contain other edges). Together, the sums
$$
\sum_{x \in e \backslash\{v\}} d_2(x)+\sum_{A \subseteq e,|A| \geq 2} d_{A \uparrow|A|+1}
$$

count each $y$ with the property that the choice of $y$ causes the removal of $e$ from the count $d_{\ell}(v)$ at least once and at most $O(1)$ many times. The number of $y$ that are counted more than once in the first sum is at most $\binom{\ell-1}{2} C_{2,2 \rightarrow 1}$. \\

Let $y_u$ be the change in $Z_{\ell}^{-}(v)$ if $v_i=u$. By definition \eqref{eq:DefinitionOfZ}, we have
\[
y_u = \begin{cases}
    0 & \mbox{ if } u \in B(v) \cup \{v\}\\
    \Delta \Big[ d_{\ell}^-(v | u) -s_\ell^- - D^\frac{\ell-1}{r-1} \phi^\delta f_\ell \Big] & \mbox{ otherwise}
\end{cases}
\]
where $\Delta d_{\ell}^-(v | u)$ above is interpreted as what the change in $d_{\ell}^-(v)$ would be if we chose $v_i=u$. With this in mind we have
\[
    \Delta d_{\ell}^-(v | u)  = d_{\{u, v\} \uparrow \ell} + \sum_{k=1}^{\ell-1} c_{\ell, k+1 \rightarrow k}(v, u).
\]
Indeed, the first term accounts for all the edges containing both $u$ and $v$, and the sum accounts for the edges in $d_\ell(v)$ which get removed due to the fact that they would contain smaller edges after choosing $v_i=u$.
Given this, we have that $\mathbb{E}\left[\Delta Z_{\ell}^{-}(v)\right]$ is equal to
\begin{align}
   &  \frac{1}{|V(i)|} \sum_{u \in V(i)} y_u   = \frac{1}{|V(i)|} \sum_{u \in V(i) \setminus B(v) \setminus \{v\}} \Delta \Big[ d_{\ell}^-(v | u) -s_\ell^- - D^\frac{\ell-1}{r-1} \phi^\delta f_\ell \Big] \nn \\
   = &\frac{1}{|V(i)|} \sum_{u \in V(i) \setminus B(v) \setminus \{v\}} \sbrac{ d_{\{u, v\} \uparrow \ell} + \sum_{k=1}^{\ell-1} c_{\ell, k+1 \rightarrow k}(v, u) -\frac{D^\frac{1}{r-1}}{N}\rbrac{ (s_\ell^-)'  +D^\frac{\ell-1}{r-1}\phi^\delta f_\ell'}  +O\left(\frac{D^\frac{2}{r-1}} {N^{2}} \rbrac{(s_\ell^-)^{\prime \prime}+ \phi^\delta f_\ell^{\prime \prime}}\right)} \nn \\
   = &\frac{1}{|V(i)|} \sum_{u \in V(i) \setminus B(v) \setminus \{v\}} \sbrac{ d_{\{u, v\} \uparrow \ell} + \sum_{k=1}^{\ell-1} c_{\ell, k+1 \rightarrow k}(v, u) -\frac{(\ell-1)s_\ell \cdot s_2 }{Nq} - \frac{D^\frac{\ell}{r-1} \phi^\delta}{N}f_\ell' + O\left( D^\frac{\ell+1}{r-1}N^{-2} \phi^{-\lambda} \right)},  \label{eq:DeltaZminus1}
\end{align}
where on the last line we used \eqref{eq:sellpmdef} to get $(s_\ell^-)'$ and Lemma \ref{lem:deterministicbounds} to simplify the big-O term. Now 
\be \label{eq:DeltaZminus2}
\sum_{u \in V(i) \setminus B(v) \setminus \{v\}}  d_{\{u, v\} \uparrow \ell} = O(d_\ell(v)) = O\rbrac{D^\frac{\ell-1}{r-1}}=O\rbrac{D^\frac{\ell}{r-1}\phi}.
\ee
The last expression follows from our assumption that $D > \phi^{-r}$, since $D^\frac{\ell-1}{r-1} < D^\frac{\ell}{r-1}\phi^{\frac{r}{r-1}} < D^\frac{\ell}{r-1}\phi$.
since each edge of $d_{\ell}(v)$ is counted for at most $\ell-1$ vertices $u$. Also, as we will explain,
\begin{align}
  & \sum_{u \in V(i) \setminus B(v) \setminus \{v\}}  \sum_{k=1}^{\ell-1} c_{\ell, k+1 \rightarrow k}(v, u)\nn\\
  &=\sum_{k=1}^{\ell-1} \sum_{u \in V(i) \setminus \{v\}}   c_{\ell, k+1 \rightarrow k}(v, u) + O\rbrac{D^\frac1{r-1}\phi \cdot D^{\frac{\ell-1}{r-1}}\phi^{1-(4 r-4) \lambda}} \nn\\
    &= \sum_{e \in d_\ell(v)} \sum_{w \in e \setminus \{v\}} d_2(w) + O\rbrac{D^{\frac{\ell}{r-1}}\phi^{1-(2 r-2\ell+2)\lambda}}\nn\\
    & = \sum_{e \in d_\ell(v) \setminus d_\ell'(v)} \sum_{w \in e \setminus \{v\}} d_2(w) + {O\rbrac{D^{\frac{\ell}{r-1}}\phi^{1-(6 r-2\ell - 4)\lambda}}}\nn\\
    &\subseteq (\ell-1)\rbrac{ s_\ell \pm 2D^\frac{\ell-1}{r-1}\phi^\delta f_\ell +O\rbrac{D'_\ell}} \rbrac{ s_2 \pm 2D^\frac{1}{r-1}\phi^\delta f_2}+ {O\rbrac{D^{\frac{\ell}{r-1}}\phi^{1-(6 r-2\ell - 4)\lambda}}}\nn\\
    &\subseteq (\ell-1)\rbrac{ s_\ell \pm 2D^\frac{\ell-1}{r-1}\phi^\delta f_\ell } \rbrac{ s_2 \pm 2D^\frac{1}{r-1}\phi^\delta f_2}+ {O\rbrac{D^{\frac{\ell}{r-1}}\phi^{1-(6 r-2\ell - 4)\lambda}}}\label{eq:DeltaZminus3}
\end{align}

Indeed, on the second line we used Conditions \ref{cond:Bbound} and \ref{eq:DynamicCoDegBound}. To justify the third line, note that for $k=1$ we have 
\[
\sum_{u \in V(i) \setminus \{v\}}   c_{\ell, 2 \rightarrow 1}(v, u) = \sum_{e \in d_\ell(v)} \sum_{w \in e \setminus \{v\}}  d_2(w) 
\]

and to take care of $k \ge 2$ we note that
\[
\sum_{k=2}^{\ell-1} c_{\ell, k+1 \rightarrow k}(v, u) = O\rbrac{d_\ell(v) \sum_{k=2}^{\ell-1} D_{k \uparrow k+1}} = O\rbrac{D^{\frac{\ell}{r-1}}\phi^{1-(2 r-2\ell+2)\lambda}}.
\]
To justify the fourth line, note that the dropped terms contribute $O(D'_\ell \cdot D_{1 \uparrow 2}) = O(D^\frac{\ell}{r-1} \phi^{1-(6r-2\ell-4)\lambda}).$ The fifth line follows from our bounds on $d_\ell(v)$ and $d_2(w)$, and the last line is from our bound on $D'_\ell$.\\

 Picking back up from \eqref{eq:DeltaZminus1} and using \eqref{eq:DeltaZminus2}, \eqref{eq:DeltaZminus3}, and \eqref{eq:DynamicVBound} we have that $\mathbb{E}\left[\Delta Z_{\ell}^{-}(v)\right]$ is at most (explanation follows) 

\begin{align}
     &\frac{(\ell-1)\rbrac{ s_\ell + 2D^\frac{\ell-1}{r-1}\phi^\delta f_\ell } \rbrac{ s_2 + 2D^\frac{1}{r-1}\phi^\delta f_2}+ {O\rbrac{D^{\frac{\ell}{r-1}}\phi^{1-(6 r-2\ell - 4)\lambda}}}}{Nq - N\phi^\delta f_V}\nn\\
    &\qquad \qquad - \frac{|V(i)\setminus B(v) \setminus \{v\}|}{|V(i)|}\sbrac{\frac{(\ell-1)s_\ell \cdot s_2 }{Nq} + \frac{D^\frac{\ell}{r-1} \phi^\delta}{N}f_\ell' + O\left( D^\frac{\ell+1}{r-1}N^{-2} \phi^{-\lambda} \right)}\nn\\
    & \le \frac{2(\ell-1) \sbrac{D^\frac{\ell-1}{r-1}\phi^\delta s_2f_\ell + D^\frac{1}{r-1}\phi^\delta s_\ell f_2}}{Nq } + \frac{(\ell-1)\phi^\delta s_\ell s_2 f_V}{Nq^2} -\frac{D^\frac{\ell}{r-1} \phi^\delta f_\ell'}{N} + O\rbrac{D^\frac{\ell}{r-1} N^{-1} \phi^{2 \delta - 5 \lambda}}\label{eq:ZminusL2}\\
    & = \frac{D^\frac{\ell}{r-1}\phi^\delta}{N}\bigg[2(\ell-1)\binom{r-1}{\ell-1} t^{r-\ell} q^{\ell-2} f_2+2(\ell-1)(r-1) t^{r-2} f_{\ell} \label{eq:ZminusL3}\\
    &\qquad \qquad \qquad\qquad \qquad +(\ell-1)(r-1)\binom{r-1}{\ell-1} t^{2 r-\ell-2} q^{\ell-2} f_V-f_{\ell}^{\prime}\bigg] +O\rbrac{D^\frac{\ell}{r-1} N^{-1} \phi^{2 \delta - 5 \lambda}}.\nn
\end{align}
Indeed, to get line \eqref{eq:ZminusL2} we use Lemma \ref{lem:est}, Lemma \ref{lem:deterministicbounds}, Condition \eqref{cond:Bbound} and line \eqref{eq:LowerBoundonN}. Line \eqref{eq:ZminusL3} is just algebra. Now observe that \eqref{eq:ZminusL3} is negative since $\phi^\delta$ is much larger than $\phi^{2\delta - 5 \lambda}$ and the quantity in square brackets is \[2(\ell-1)\binom{r-1}{\ell-1} t^{r-\ell} q^{\ell-2} f_2+2(\ell-1)(r-1) t^{r-2} f_{\ell}  +(\ell-1)(r-1)\binom{r-1}{\ell-1} t^{2 r-\ell-2} q^{\ell-2} f_V-f_{\ell}^{\prime} = -\Omega(1)\]
by Lemma \ref{lem:deterministicbounds} parts (4), (5) and (6). Therefore $Z_\ell^-$ is a supermartingale. Similarly, $\Zs_\ell^-$ is a submartingale.

\subsection{Applying martingale variation inequalities}

We complete the proof by applying martingale variation inequalities to prove that $Z_V$ and $Z_{\ell}^{ \pm}$ remain negative a.a.s. We will apply the following lemmas (which both follow from Hoeffding \cite{hoef}):

\begin{lemma}\label{symmetricAH} Let $X_i$ be a supermartingale such that 
$|\Delta X| \leq c_i$ for all $i$. Then $$ \P(X_m - X_0 > d) \leq  \exp\left(-\frac{d^2}{2 \displaystyle\sum_{i\leq m} c_i^2 }\right)$$ \end{lemma}

\begin{lemma}\label{asymmetricAH} Let $X_i$ be a supermartingale such that $-B \leq \Delta X \leq b$ for all $i$, for some $b < \frac{B}{10}$. Then for any $d < bm$ 
we have $$ \P(X_m -X_0 > d) \leq \exp \left(- \frac{d^2}{3 m bB }\right)$$
\end{lemma}

We apply Lemma \ref{symmetricAH} to the supermartingale $Z_V(i)$. Note that if the vertex $v$ is inserted to the independent set at step $i$ then $\Delta V(i)=-1-d_2(v).$ Now we have
$$
\begin{aligned}
 |\Delta Z_V| &\le |\Delta V(i)| + |\Delta N q| + |\Delta N \phi^{\delta} f_V|= O\rbrac{D_{1 \uparrow 2} + D^\frac{1}{r-1} q' + D^\frac{1}{r-1} \phi^{\delta} f_V'} = O\rbrac{D^\frac{1}{r-1} \phi^{-(2r-4)\lambda}} \eqcom 
\end{aligned}
$$
where the last equality holds by Lemma \ref{lem:deterministicbounds}.\\

As we have that $Z_V(0)=-N \phi^{\delta}$, by definition, the probability that $Z_V$ is positive at time step $\tau$ is at most

$$
\begin{aligned}
\exp \left\{-\Omega\left(\frac{\left(N \phi^{\delta}\right)^2}{N D^{-\frac{1}{r-1}} \log^\frac{1}{r-1}(1/\phi) \left[D^\frac{1}{r-1} \phi^{-(2r-4)\lambda}\right]^2}\right)\right\}& \leq \exp \left\{-\Omega\left(ND^{-\frac{1}{r-1}} \phi^{2\delta + (4r-8)\lambda} \log^{-\frac{1}{r-1}}(1/\phi)\right)\right\} \\
&\leq \exp \left\{- \Omega\rbrac{\phi^{-1+2\delta + (4r-8)\lambda} \log^{-\frac{1}{r-1}}(1/\phi)} \right\} =o(1),
\end{aligned}
$$
where on the second line we used $N=\Omega\left(D^{\frac{1}{r-1}}\phi^{-1}\right)$ from \eqref{eq:LowerBoundonN}.\\

 For our upper bound on $d_{\ell}^{+}(v)$ we apply \ref{asymmetricAH} to the supermartingale $Z_{\ell}^{+}(v)$. Note that we have either $\Delta Z_{\ell}^{+}(v)=0$ due to the supermartingale beng frozen, or else we can upper bound $\Delta Z_{\ell}^{+}(v)$ as follows:
\be\label{eq:DeltaZ_ell_UB}
\Delta Z_{\ell}^{+}(v) = \Delta d^+_\ell(v) -\Delta s^+_\ell -\Delta D^\frac{\ell-1}{r-1} \phi^\delta f_\ell  \le D_{2 \uparrow \ell+1}  = O(D^\frac{\ell-1}{r-1} \phi^{1-(2r-4)\lambda}).
\ee
Note that above, we used the fact that $s^+_\ell, f_\ell$ are increasing.
For a lower bound on $\Delta Z_{\ell}^{+}(v)$, using $\Delta d^+_\ell(v) \ge 0$ and Taylor's theorem, we obtain
\be \label{eq:DeltaZ_ell_LB}
\Delta Z_{\ell}^{+}(v)\geq -O\rbrac{\left[\left(s_{\ell}^{+}\right)^{\prime} +D^{\frac{\ell-1}{r-1}}f_\ell'\right] \frac{D^{\frac{1}{r-1}}}{N}  }\ge -O\left(\frac{D^{\frac{\ell}{r-1}}}{N}\right)
\ee

We have $Z_{\ell}^{+}(0)=-D^{\frac{\ell-1}{r-1}}\phi^{\delta}$. For our present application of Lemma \ref{asymmetricAH} we let $B=O\left(\frac{D^{\frac{\ell}{r-1}}}{N}\right), b=D^\frac{\ell-1}{r-1} \phi^{1-(2r-4)\lambda}$ (from \eqref{eq:DeltaZ_ell_UB}, \eqref{eq:DeltaZ_ell_LB}), $d=D^{\frac{\ell-1}{r-1}}\phi^\delta$ and $m=i_{\max}$. The assumptions of Lemma \ref{asymmetricAH} hold since $\frac{D^{\frac{\ell}{r-1}}}{N}=o\left(D^\frac{\ell-1}{r-1} \phi^{1-(2r-4)\lambda}\right)$ using that \eqref{eq:LowerBoundonN} and $r \geq 3$, and $D^{\frac{\ell-1}{r-1}}\phi^{\delta}=o\left(\frac{D^{\frac{\ell}{r-1}}}{N} \cdot i_{\max }\right)$. So the probability that $Z_{\ell}^{+}(v)$ is positive at step $\tau$ is at most

$$
\begin{aligned}
& \exp \left\{-{\Omega}\left(\frac{\left(D^{\frac{\ell-1}{r-1}}\phi^{\delta}\right)^2}{N D^{-\frac{1}{r-1}} \log^\frac{1}{r-1}(1/\phi)\cdot \frac{1}{N} D^{\frac{\ell}{r-1}} \cdot D^\frac{\ell-1}{r-1} \phi^{1-(2r-4)\lambda}}\right)\right\} = \exp \cbrac{ -\Omega \rbrac{\phi^{-1 + 2\delta + (2r-4)\lambda} \log^{-\frac{1}{r-1}}(1/\phi)} },
\end{aligned}
$$
 which is small enough to beat a union bound over $O(N)$ choices for $v$ and $\ell$. The lower bound on $d_{\ell}^+(v)$ follows similarly. For our upper bound on $d_{\ell}^{-}(v)$ we apply Lemma \ref{asymmetricAH} to the supermartingale $Z_{\ell}^{-}$. Note that
\be
-O\left(\frac{D^{\frac{\ell}{r-1}}}{N}\right)<\Delta Z_{\ell}^{-}(v)<O\left(\sum_{1 \leq k \leq \ell-1} C_{\ell, k+1 \rightarrow k}\right)=O\left(D^{\frac{\ell-1}{r-1}} \phi^{1-(4r-8)\lambda}\right)
\ee
where the lower bound argument follows similarly to \eqref{eq:DeltaZ_ell_LB}, and the upper bound accounts for the number of edges, which, after adding one of their vertices to the independent set, is contained in an edge of size $\ell$ containing $v$. We have $Z_{\ell}^{-}(0)=-D^{\frac{\ell-1}{r-1}}\phi^{\delta}$. The rest of this application of Lemma \ref{asymmetricAH} is very similar to the previous one. We let $B=O\left(\frac{D^{\frac{\ell}{r-1}}}{N}\right), b=D^\frac{\ell-1}{r-1} \phi^{1-(4r-8)\lambda}$, $d=D^{\frac{\ell-1}{r-1}}\phi^\delta$ and $m=i_{\max}$. The assumptions of Lemma \ref{asymmetricAH} hold since $\frac{D^{\frac{\ell}{r-1}}}{N}=o\left(D^\frac{\ell-1}{r-1} \phi^{1-(4r-8)\lambda}\right)$  by \eqref{eq:LowerBoundonN} and the fact that $r \geq 3$, and $D^{\frac{\ell-1}{r-1}}\phi^{\delta}=o\left(\frac{D^{\frac{\ell}{r-1}}}{N} \cdot i_{\max }\right)$. So the probability that $Z_{\ell}^{-}(v)$ is positive at step $\tau$ is at most

$$
\begin{aligned}
& \exp \left\{-{\Omega}\left(\frac{\left(D^{\frac{\ell-1}{r-1}}\phi^{\delta}\right)^2}{N D^{-\frac{1}{r-1}} \log^\frac{1}{r-1}(1/\phi)\cdot \frac{1}{N} D^{\frac{\ell}{r-1}} \cdot D^\frac{\ell-1}{r-1} \phi^{1-(4r-8)\lambda}}\right)\right\} = \exp \cbrac{ -\Omega \rbrac{\phi^{-1 + 2\delta + (4r-8)\lambda} \log^{-\frac{1}{r-1}}(1/\phi)} },
\end{aligned}
$$
 which is small enough to beat a union bound over $O(N)$ choices for $v$ and $\ell$. The lower bound on $d_{\ell}^-(v)$ follows similarly.

\section{Concluding remarks}
There are quite a few natural next directions to follow. Here we consider the case in which $w$ is small enough relative to $n$ such that the matching sunflowers, and those with relatively small kernels, are the dominant type. Thus, it is of interest to see what happens when this is not the case. In fact, when $\frac{n}{w}<r$ any existing sunflower must have a non-empty kernel, so perhaps the regime where $w \in (\frac{n
}{r}, n-1]$ would be a natural next step.
In addition, the algorithm described in \cite{abbott1992set} includes an additional backtracking step. The authors use a basic learning algorithm to help determine which edges should be removed in the backtracking step, and also use it to bias the algorithm towards choosing specific sets. Thus, it would be interesting if one could give a proof related to this more involved algorithm.
\subsection*{Acknowledgment}

The authors would like to thank Nick Christo for bringing them together.
\bibliographystyle{alpha}
\bibliography{bib.bib}

\appendix

\section{The sunflower-free process} \label{sec:SF-free}
We use this section to show that the Sunflower Hypergraph satisfies the assumptions of Theorem \ref{thm:ind}. Again, recall that $\hnw$ is the hypergraph whose vertex set $\mathscr{V}_{n,w}$ is the set of all possible $\binom{n}{w}$ sets in the $w$-uniform complete hypergraph $\mathcal{K}_{n,w}$, and the edge set  $\mathscr{E}_{w,r}$ is the  set of all possible $r$-sunflowers in $\mathcal{K}_{n,w}$ where each edge in $\mathscr{E}_{w,r}$ is an $r$-sunflower in $\mathcal{K}_{n,w}$. Firstly notice that $\hnw$ is $r$-uniform and $D$-regular, where 
\be \label{eq:SF-Ddef}
D :=\sum\limits_{s=0}^{w-1}\binom{w}{s}\frac{1}{(r-1)!}\prod\limits_{i=1}^{r-1}\binom{n-i w+(i-1)s}{w-s}= \sum\limits_{s=0}^{w-1} \frac{w!(n-w)!}{(r-1)! s![(w-s)!]^{r} [(n-rw+(r-1)s]! }.
\ee
In the equation displayed above, $s$ represents the size of the kernel. Since $\hnw$ is uniform and regular, Conditions \eqref{cond:almostregular} and \eqref{cond:max1degree} are satisfied for any $\phi >0$. We will let 
\be\label{eq:SF-phidef}
\phi:= \exp \cbrac{-\frac{w^2}{10n}}
\ee
 Since we assume $w= n^{\alpha}$  for $\frac12 < \alpha < 1$, we have $\phi=o(1).$ In this regime, we will see that the ``matching sunflowers'', or the sunflowers with empty kernels, contribute the most to the degree.
We now estimate $N$ and $D$.
\begin{claim}\label{clm:SF-NDest}
    We have
    \be\label{eq:SF-NDest}
    \begin{split}
         N & = \operatorname{exp}\left\{w+ w\log\left(\frac{n}{w}\right) -\frac{w^2}{2n} + o\left( \frac{w^2}{n}\right) \right\},\\
       D & = \operatorname{exp}\left\{w(r-1) + w(r-1)\log\left(\frac{n}{w}\right) -\frac{w^2(r^2-1)}{2n} +o\left( \frac{w^2}{n}\right)\right\}.
    \end{split}
    \ee
\end{claim}

Note that this claim implies that $D > \phi^{-r}$ as required in the assumptions of Theorem \ref{thm:ind}.

\begin{proof}[Proof of Claim \ref{clm:SF-NDest}]
Recall Stirling's approximation, which says
\be
n! \approx \sqrt{2 \pi n} \left( \frac{n}{e} \right)^n.
\ee
Thus we have
\be\label{eq:approximationofN}
\begin{split}
    N  &= \binom nw = \frac{n!}{w! (n-w)!} = \frac{\left( \frac{n}{e} \right)^n}{ \left( \frac{w}{e} \right)^w  \left( \frac{n-w}{e} \right)^{n-w}} \exp\{O(\log n)\}\\
    &=\operatorname{exp}\left\{n \log n - w \log w - (n-w) \log(n-w) + O(\log n)\right\}\\
&=\operatorname{exp}\left\{n \log n - w \log w - (n-w) \sbrac{\log n - \frac wn - \frac{w^2}{2n^2}+O\rbrac{\frac{w^3}{n^3}}} + O(\log n)\right\}\\
&=\operatorname{exp}\left\{w+ w\log\left(\frac{n}{w}\right) -\frac{w^2}{2n} + o\left( \frac{w^2}{n}\right) \right\},
\end{split}
\ee
where on the third line we used the fact that when $y = o(x)$ we have $\log(x+y) = \log x + \frac{y}{x} - \frac{y^2}{2x^2} + O\rbrac{\frac{y^3}{x^3}}.$ This follows from the Taylor series  $\log(1+t) = t - \frac 12 t^2 + \cdots $.
Thus our estimate of $N$ is justified. Now we estimate $D$.

 Let $D_s$ be the term corresponding to a particular value of $s$ in the sum from \eqref{eq:SF-Ddef}. We bound the following ratio:
\be \label{eq:R(s)}
\begin{split}
R(s):=&\frac{D_{s+1}}{D_s} = \frac{\binom{w}{s+1} \frac{(n-w)_{(r-1)(w-s-1)}}{[(w-s-1)!]^{r-1} }}{\binom{w}{s} \frac{(n-w)_{(r-1)(w-s)}}{[(w-s)!]^{r-1} }}\\
& = \frac{(w-s)^r}{(s+1)\bigg(n-w-(r-1)(w-s)\bigg)_{r-1}}
\end{split}
\ee
where we use $(x)_n\coloneqq x(x-1)(x-2) \dots (x-n+1)$ to denote the falling factorial.
We claim that  
\[
\max\{D_s: 0 \le s \le w-1\} = \begin{cases} D_0 & \text{ if } \alpha \le \frac{r-1}{r},\\
D_{s^*} \text{ for some } s^* \approx n^{r\alpha - r + 1} & \text{ otherwise.}
\end{cases}
\]
Indeed, \eqref{eq:R(s)} shows that for fixed $r, n, w$ and letting $s$ be a real number, $R(s)$ is positive and strictly decreasing on $[0, w]$ and $R(w)=0$. If $\alpha \le \frac{r-1}{r}$ then $R(0) \le 1$, justifying the first case above. If $1> \alpha > \frac{r-1}{r}$ then the maximum $D_s$ occurs at the largest integer value $s$ such that $R(s-1)>1$. But we can see that there is a real solution $s$ to the equation $R(s)=1$, where $s \approx \frac{w^r}{n^{r-1}} = n^{r\alpha - r + 1}$. For that it helps to observe that this value of $s$ is much smaller than $w=n^\alpha$ since $\alpha < 1$ and $r \ge 3$. This justifies the second case above. 

We now estimate the largest possible $D_s$. Let $s^*$ be the value of $s$ maximizing $D_s$, so $s^*=0$ if $\alpha \le \frac{r-1}{r}$ and otherwise $s^* \approx n^{r\alpha - r + 1}.$ Then (using the convention $0 \log 0 = 0$),

\begin{align*}
D_{s^*} &= \frac{w!(n-w)!}{(r-1)! s^*![(w-s^*)!]^{r} [(n-rw+(r-1)s^*]! }\\
&= \frac{\left( \frac{w}{e} \right)^w\left( \frac{n-w}{e} \right)^{n-w}}{\left( \frac{s^*}{e} \right)^{s^*}\left( \frac{w-s^*}{e} \right)^{r(w-s^*)} \left( \frac{n-rw+(r-1)s^*}{e} \right)^{n-rw+(r-1)s^*}} \exp\{O(\log n)\}\\ 
&= \exp\bigg\{ w \log w + (n-w) \log (n-w) - s^* \log s^* - r(w-s^*) \log(w-s^*)\\
&  \qquad - (n-rw+(r-1)s^*) \log [n-rw+(r-1)s^*] + O(\log n)\bigg\}\\
& = \exp\bigg\{ w \log w + (n-w) \sbrac{\log n - \frac wn - \frac{w^2}{2n^2} +O\rbrac{\frac{w^3}{n^3}}}  - r(w-s^*) \sbrac{\log w - O\rbrac{\frac {s^*}{w}} }\\
&  \qquad - (n-rw+(r-1)s^*) \sbrac{\log n - \frac {rw-(r-1)s^*}n - \frac{[{rw-(r-1)s^*}]^2}{2n^2} +O\rbrac{\frac{w^3}{n^3}}} + o\left( \frac{w^2}{n}\right)\bigg\}\\
&=\operatorname{exp}\left\{w(r-1) + w(r-1)\log\left(\frac{n}{w}\right) -\frac{w^2(r^2-1)}{2n} +o\left( \frac{w^2}{n}\right)\right\}.
\end{align*}

Now $D_{s^*} \le D \le wD_{s^*}$, and our estimate of $D_{s^*}$ already has a multiplicative error $\exp\{o(w^2/n)\}$ which can absorb the $w$ in the upper bound for $D$. Thus the lower and upper bound for $D$ are both of the form claimed.

\end{proof}

Our next claim verifies Condition \eqref{cond:maxLdegree}. Note that since $\hnw$ is $r$-uniform, it suffices to check the case $k=r$.
\begin{claim}\label{clm:SF-Delta}
    For $2 \le \ell \le r-1$ we have
    \[
    \Delta_{\ell}(\hnw) \le \phi D^\frac{r-\ell}{r-1}
    \]
\end{claim}

\begin{proof}

To bound $\Delta_{\ell}(\hnw),$ we fix $\ell$ sets in $\binom{[n]}{w}$ and bound the number of $r$-sunflowers containing the fixed sets. We assume our fixed sets form an $\ell$-sunflower, since otherwise there is no $r$-sunflower containing them. Letting $s \le w-1$ be the size of the kernel of our fixed $\ell$-sunflower, the number of $r$-sunflowers containing it is 

\be
  \prod\limits_{i=\ell}^{r-1}\binom{n-i w+(i-1)s}{w-s} =  \frac{(n-\ell w+(\ell-1)s)!}{[(w-s)!]^{r-\ell}(n-rw+(r-1)s)!}.
\ee
Letting the above expression be $g(s)$, we see that for $s \le w-2$ we have
\[
\frac{g(s+1)}{g(s)} = \frac{1}{(w-s-1)^{r-\ell}} \cdot \frac{(n-\ell w+(\ell-1)(s+1))_{\ell-1}}{(n-r w+(r-1)(s+1))_{\ell-1}} = \frac{1+o(1)}{(w-s-1)^{r-\ell}}.
\]
Now for $s=w-2$ the above expression is $1+o(1)$, while for $s \le w-3$ it is strictly less than 1. Thus the maximum possible value of $g(s)$ is $\approx g(0)$.

Now $g(0)$ is

\be \label{eq:SF-DeltaL}
\begin{split}
    &\frac{(n-\ell w)!}{(w!)^{r-\ell}(n-rw)!}\\
    &=\operatorname{exp}\left\{ (n-\ell w)\log(n- \ell w) - (r-\ell) w \log w - (n-rw) \log(n-rw)\right\}\\
     &=\operatorname{exp}\left\{w(r-\ell) + w(r-\ell)\log\left(\frac{n}{w}\right) -\frac{w^2(r^2-\ell^2)}{2n}+ o\left(\frac{w^2}{n}  \right)\right\} \eqpd\\
\end{split}
\ee
Now using \eqref{eq:SF-NDest} and \eqref{eq:SF-DeltaL}, we have
\be
\begin{split}
&\frac{\Delta_{\ell}\left(\hnw \right)}{D^{\frac{r-\ell}{r-1}}} \approx \frac{\operatorname{exp}\left\{w(r-\ell) + w(r-\ell)\log\left(\frac{n}{w}\right) -\frac{w^2(r^2-\ell^2)}{2n}+ o\left(\frac{w^2}{n}  \right)\right\}}{\operatorname{exp}\left\{w(r-\ell) + w(r-\ell)\log\left(\frac{n}{w}\right) -\frac{w^2(r+1)}{2(r-\ell) n}+ o\left(\frac{w^2}{n}  \right)\right\}}\\
&= \exp\cbrac{-\frac{w^2}{2n} \rbrac{r^2 - \ell^2 - \frac{r+1}{r-\ell}+o(1)}} \le \phi,
\end{split}
\ee
where the last inequality is because
\[
r^2 - \ell^2 - \frac{r+1}{r-\ell} = (r-\ell)(r+ \ell) - \frac{r+1}{r-\ell} \ge 1 \cdot (r+\ell) - \frac{r+1}{1} = \ell-1 \ge 1.
\]
This justifies the claim.

\end{proof}

Our next claim involves some positive constant $\e$. It suffices to take $\e = 1/2$. 
\begin{claim}\label{clm:SF-largeintersection}
    Consider two vertices $W_0$ and $W_1$ of $\hnw$, and assume that $|W_0 \cap W_1|\le (1-\e)w$. Then the $(r-1)$-codegree of $W_0, W_1$ is at most $\phi D$.
\end{claim}
\begin{proof}
 The $(r-1)$-codegree of $W_0, W_1$ is the number of choices for sets $W_2, \ldots W_{r}$ such that both $\{W_0, W_2, \ldots W_{r}\}$ and $\{W_1, W_2, \ldots W_{r}\}$ are sunflowers. Note that the sets $W_2, \ldots, W_r$ cannot intersect $W_1 \setminus W_0$ or $W_0 \setminus W_1$, and so our two sunflowers must have a common kernel $K \subseteq W_0 \cap W_1$. Letting $t:=|W_0 \cap W_1|\le (1-\e)w$ and $s=|K|\le t$, the number of choices for $W_2, \ldots W_{r}$ is
\begin{align}\label{eq:SF-codegree1}
    \begin{split}
     \sum_{s=0}^{t} \binom ts \prod\limits_{i=1}^{r-1} \binom{n-2w+t - (i-1)(w-s)}{w-s} = \sum_{s=0}^{t}  \frac{t!(n-2w+t)!}{s!(t-s)![(w-s)!]^{r-1}[(n+t-(r+1)w+(r-1)s]!}\\
    \end{split}.
\end{align}
The above is increasing in $t$ (easier to see if we look at the expression on the left),  so we will henceforth assume that $t=(1-\e)w$. Let $F_s$ be the $s$ term of the sum above. Then (re-defining the function $R(s)$ so it no longer means the same as it did in the proof of Claim \ref{clm:SF-NDest})
\begin{align*}
   R(s):= \frac{F_{s+1}}{F_s} = \frac{(t-s)(w-s)^{r-1}}{(s+1)\sbrac{n+t-(r+1)w+(r-1)(s+1)}_{r-1}}.
\end{align*}
Similar to the proof of Claim \ref{clm:SF-NDest}, we can see that the maximum possible value of $F_s$ occurs when $s=s^*$ where $s^*= 0$ if $\alpha \le \frac{r-1}{r}$ and $s^* \approx \frac{tw^{r-1}}{n^{r-1}} = (1-\e)n^{r\alpha - r + 1}$ if $\alpha > \frac{r-1}{r}$. Now

\begin{align}
       F_{s^*} & = \frac{t!(n-2w+t)!}{s^*!(t-s^*)![(w-s^*)!]^{r-1}[n+t-(r+1)w+(r-1)s^*]!}\nn\\
       & = \frac{\rbrac{\frac{t}{e}}^{t} \rbrac{\frac{n-2w+t}{e}}^{n-2w+t}}{\rbrac{\frac{s^*}{e}}^{s^*} \rbrac{\frac{w-s^*}{e}}^{w-s^*} \rbrac{\frac{n+t-(r+1)w+(r-1)s^*}{e}}^{n+t-(r+1)w+(r-1)s^*}} \exp\{O(\log n)\}\nn\\
       & = \exp \bigg\{ t \log t + (n-2w+t) \log (n-2w+t) - s^* \log s^* - (w-s^*) \log (w-s^*)\nn\\
       & \qquad - [n+t-(r+1)w+(r-1)s^*)] \log[n+t-(r+1)w+(r-1)s^*] +O(\log n)\bigg\}\nn\\
       & = \exp \bigg\{ t \log t + (n-2w+t) \sbrac{\log n - \frac{2w-t}{n} - \frac{(2w-t)^2}{2n^2} + O\rbrac{\frac{w^3}{n^3}}} - s^* \log s^*\nn \\
       & \qquad - (w-s^*) \sbrac{\log w + O\rbrac{\frac{s^*}{w}}}- [n+t-(r+1)w+(r-1)s^*)] \bigg[ \log n - \frac{(r+1)w-t-(r-1)s^*}{n} \nn\\
       & \qquad -\frac{[(r+1)w-t-(r-1)s^*]^2}{2n^2} +O\rbrac{\frac{w^3}{n^3}} \bigg]+O(\log n)\bigg\}\nn\\
&=\operatorname{exp}\left\{w(r-1)+ w(r-1)\log\left(\frac{n}{w}\right) -\frac{w^2(r^2+2\e(r-1)-1)}{2n} +o\left(\frac{w^2}{n}  \right)\right\}.\nn
\end{align}

This completes the proof of the claim.

\end{proof}
Our next claim verifies Condition \eqref{cond:Bbound}.
\begin{claim}\label{clm:SF-bbound}
        For a fixed vertex $W \in \mathscr{E}_{w,r}$,  we have
        $ |B(W)| \leq \phi D^\frac{1}{r-1}.$
        
\end{claim}

\begin{proof}
By Claim \ref{clm:SF-largeintersection}, $|B(W)|$ is at most the number of $w$-sets $W'$ with $|W \cap W'|> (1-\e)w$. Using the bound $\binom ab \le \rbrac{\frac {ea}{b}}^b$, we have that 
\[
   |B(W)| \le  \binom{w}{(1-\e)w}\binom{n}{\e w} =\binom{w}{\e w}\binom{n}{\e w}\le \rbrac{\frac {ew}{\e w}}^{\e w} \rbrac{\frac {en}{\e w}}^{\e w} = \exp \cbrac{\e w \log \rbrac{\frac nw} + O(w)}.
\]
Now using Claim \ref{clm:SF-NDest}, we have
\[
\frac{|B(W)|}{D^\frac{1}{r-1}} \le \frac{\exp \cbrac{\e w \log \rbrac{\frac nw} + O(w)}}{\operatorname{exp}\left\{w + w\log\left(\frac{n}{w}\right) -\frac{w^2(r+1)}{2n} +o\left(\frac{w^2}{n} \right)\right\}} \le \exp\cbrac{-(1-\e)w \log \rbrac{\frac nw} + O(w)} = o(\phi),
\]
so the claim holds with room to spare.

\end{proof}
Now we verify Condition \eqref{cond:ddubprime}. Since $\hnw$ is $r$-regular it suffices to check the case $b=r$.
\begin{claim}
    $d^{\prime\prime}_{A \uparrow r}(\hnw) < \phi D^{\frac{r-|A|}{r-1}}$ for all $A \subseteq V(\hnw)$ with $1 \le |A| < r$.
\end{claim}

\begin{proof}
    Fix our set $A$, with $1 \le |A| < r$. Suppose some $r$-sunflower contains $A$ and some bad pair $W, W' \notin A$. Then $|W \cap W'| \ge (1-\e)w$ by Claim \ref{clm:SF-largeintersection}, and so the kernel of our sunflower has size at least $(1-\e)w$. In other words, every pair of sets in this sunflower is a bad pair. Now using Claim \ref{clm:SF-bbound}, the number of such sunflowers is at most 
    \[
    \rbrac{\phi D^\frac{1}{r-1}}^{r-|A|} \le \phi D^{\frac{r-|A|}{r-1}}
    \]
    as required. 
\end{proof}

We have now checked all the conditions of Theorem \ref{thm:ind}, and so applying the theorem guarantees that the sunflower-free process gives us a family of size at least
\[
\Omega\left(N \cdot\left(\frac{\log 1/\phi}{D}\right)^{\frac{1}{r-1}}\right)  = \Omega \rbrac{(w^2/n)^\frac{1}{r-1} ND^{-\frac{1}{r-1}}}
\]
asymptotically almost surely. This completes the proof of Theorem \ref{thm:SunflowerTheorem}.

\section{Thresholds} 
For completeness, we include proofs for Propositions \ref{prop:Hnpw} and \ref{prop:SpreadnessThreshold}.

\begin{proof}[Proof of Proposition \ref{prop:Hnpw}]
We apply the first moment method to obtain the 0-statement, and the second moment method for the 1-statement. Let $X_r$ be the number of $r$-sunflowers in $\hnpw.$ Since the number of edges in $\hnw$ is $ND/r$, the expected number of $r$-sunflowers in $\hnpw$ is
    \begin{align}
    \begin{split}
        \E[X_r] &= \frac 1r \binom nw p^rD = \frac 1r N D p^r.
    \end{split}  
    \end{align}
Thus if we have $p =o\left(ND\right)^{-1/r}$ then 
\[
\Pr(X_r > 0) \le \E[X_r] =  o(1).
\]
  Now suppose $p =\omega\left(ND\right)^{-1/r}$. We have
\be \label{eq:2ndmoment1}
        \E[X_r^2] = \sum_{(S, S')} \Pr[S] \Pr[S' | S]
        \ee
        where we sum over all pairs of sunflowers $(S, S')$, $\Pr[S]=p^r$ is the probability that $S$ is in our random collection, and $\Pr[S' | S]$ is the conditional probability of $S'$ given $S$. If $|S \cap S'|=j$ (so they share $j$ many $w$-sets), then $\Pr[S' | S]=p^{r-j}$. Picking back up from \eqref{eq:2ndmoment1}, we get (explanation follows)
\begin{align}
\begin{split}
       \E[X_r^2] &= \sum_{S} p^r \left[ \sum_{j=0}^r \sum_{\substack{S' \\ |S' \cap S| =j}} p^{r-j} \right]  \\
       & \le \sum_{S} p^r \left[\frac 1r N D p^r  + rD p^{r-1} + \sum_{j=2}^{r-1} \binom rj \phi D^\frac{r-j}{r-1} p^{r-j}  + 1\right] \\
       &= (1+o(1)) \left(\frac 1r N D p^r \right)^2.
\end{split}
\end{align}
The first line follows from the discussion above. For the second line it will help to recall the facts we proved in Appendix \ref{sec:SF-free}. We bound the $j=0$ term using the fact that the total number of choices for $S'$ is at most $ND/r$. For $j=1$ there are at most $rD$ choices for $S'$ since we choose one of $r$ edges of $S$ to share, and then complete $S'$ given this one edge. For $2 \le j \le r-1$ we use Claim \ref{clm:SF-Delta} to conclude that the number of choices for $S'$ is at most $\phi D^\frac{r-j}{r-1}.$ Finally for $j=r$ of course we have only the choice $S'=S$. Now to justify the third line, we argue that the first term in the brackets dominates the rest. Indeed, the second term is negligible since $Np =\omega(1)$. For each term $2 \le j \le r-1$ we have $\phi D^\frac{r-j}{r-1} p^{r-j} = o(NDp^r)$ since $p^j = \omega(\phi D^\frac{j-1}{r-1}/N)$. Finally of course the 1 is negligible since we chose $p$ so that $N D p^r = \omega(1)$. This completes the proof of the proposition.

\end{proof}

The proof of Proposition \ref{prop:SpreadnessThreshold} uses the celebrated result of Frankston, Kahn, Narayanan, and Park \cite{FKNP} in the style presented in \cite{frieze2016introduction}.
\begin{proof}[Proof of Proposition \ref{prop:SpreadnessThreshold}]
Let $\mathcal{K}_{n,w}$ denote the set of all possible subsets of $[n]$ of size $w$, and recall that $\mathscr{H}_{n,w}$ is the hypergraph with vertex set $\binom{[n]}{w}$ and with an edge of size $r$ for each $r$-sunflower with kernel of size $t$ in $\mathcal{K}_{n,w}$. In other words, $\mathscr{H}$ is $r$-bounded and has 
\begin{equation}
    \binom{n}{t} \binom{n-t}{r(w-t)}\psi\left(r(w-t),r, w-t\right)
\end{equation}

edges where $\psi(rs, r, s) \coloneqq \frac{(rs)!}{r!(s!)^r}$. 

The core of the proof is showing that $\mathscr{H}$ $\kappa$-spread for suitable $\kappa$. Let 
\[
\langle S\rangle \coloneqq \{T: S \subseteq T \subseteq V(\mathscr{H})\}
\]
denote the subsets containing $S$.
If we choose an arbitrary subset of the vertices $S= \{x_1, x_2 , \dots x_s\}$ of size $s$ then $E(\mathscr{H}) \cap \langle S\rangle = \emptyset$ unless $s \leq r$  and $\forall i, j \in [s] \ e_i \cap e_j = T$ for some $|T|=t$, and $e_i$ and $e_j$ are the hyperedges denoted by $x_i$ and $x_j$.

Using that, for any $n$ $\sqrt{2 \pi n}\left(\frac{n}{e}\right)^n e^{\frac{1}{12 n+1}}<n!<\sqrt{2 \pi n}\left(\frac{n}{e}\right)^n e^{\frac{1}{12 n}}$ \cite{robbins}, we have
\begin{align}
    \begin{split}
        \frac{\left| E(\mathscr{H}) \cap \langle S\rangle \right|}{\left| E(\mathscr{H})\right|} &=\frac{\binom{n-(s(w-t)+t)}{(r-s)(w-t)}\psi((r-s)(w-t), (r-s), (w-t))}{\binom{n}{t} \binom{n-t}{r(w-t)}\psi\left(r(w-t),r, w-t\right)}\\
        &=\frac{(n-(s(w-t)-t)!((w-t)!)^st!r!}{n!(r-s)!}\\
        &\leq e^{s(w-t-1)}\left(\frac{(w-t)!^s(r-s)^s}{n^{s(w-t)}}\right) \left[\frac{t^{t+1/2}r^{r+1/2}\sqrt{2\pi}\cdot e^{\frac{1}{12n}}}{n^{t+1/2}(r-s)^{r+1/2}}\right]\\
        &\leq e^{s(w-t-1)}\left(\frac{(w-t)!^s(r-s)^s}{n^{s(w-t)}}\right) \left[\frac{n^{t+1/2}r^{r+1/2}\sqrt{2\pi}}{n^{t+1/2}(r-s)^{r+1/2}}\right]\\
        &\leq c^se^{s(w-t-1)}\left(\frac{(w-t)!^sr^s}{n^{s(w-t)}}\right)\\
    \end{split}
\end{align}
for some constant $c>1$.
Thus, $\mathscr{H}$ is

\begin{equation}
    \kappa = \left(\frac{n^{(w-t)}}{ce^{(w-t-1)}(w-t)!r}\right)
\end{equation}
    spread.  Moreover 

\begin{align}
    \begin{split}
        \frac{\binom{n}{w}}{\kappa} &= ce^{(w-t-1)}\left(\frac{(w-t)!r}{n^{(w-t)}}\right) \frac{n!}{w!(n-w)!}\\
        &\leq ce^{(w-t-1)}\left(\frac{w!r n^t}{n^{w}}\right) \frac{n^{n+w+\frac{1}{2}}}{w!(n-w)^{n+\frac{1}{2}}}\\
        &\leq ce^{(w-t-1)}rn^t \frac{n^{n+\frac{1}{2}}}{(n-w)^{n+\frac{1}{2}}}\\
        &\leq ce^{(2w-t)}rn^t\\
    \end{split}
\end{align}

Therefore, applying \cite{FKNP}  gives that if $m= Ke^{2w-t}rn^t\log r$ for some absolute constant $K$, then $\mathbb{H}_{n,m,w}$ contains  a sunflower with kernel size $t$ w.h.p.. In particular notice that this value is $Ke^{2w-t}r\log r$ for a matching of size $r$.
\end{proof}

\section{Analytic considerations} \label{sec:errorfunctions}
In this section we discuss our deterministic functions $q, s_\ell^{\pm}, f_V, f_\ell$ and the choice of constants. 

Before our next lemma, recall from \eqref{eq:q}, \eqref{eq:s_ell} and \eqref{eq:i_max} that 
\[
q(t)=e^{-t^{r-1}}, \qquad s_{\ell}(t)=\binom{r-1}{\ell-1} D^{\frac{\ell-1}{r-1}} t^{r-\ell} q^{\ell-1}, \qquad t_{max}= \zeta \log^\frac{1}{r-1}(1/\phi).
\]

Our error functions will have the form
\be\label{eq:settingVariation1}
\begin{gathered}
f_{\ell}=\left(1+t^{r-\ell+2}\right) q^{\ell} \exp \left(\alpha t+\beta t^{r-1}\right)   \\
f_V=\left(1+t^2\right)  q^2 \exp \left(\alpha t+\beta t^{r-1}\right) .
\end{gathered}
\ee

\begin{lemma}\label{lem:deterministicbounds}
    For any fixed $r \ge 2$, $\delta = 1/10$ and $\lambda=1/100r$ there exist positive constants $\zeta, \alpha, \beta$   satisfying the following for all $t \in [0, t_{max}]$.
        \begin{enumerate}
        \item \[f_V^{\prime}>3 f_2\]
        \item  \[f_\ell' > 5\ell q^{-1}f_{\ell+1}\]
        \item  \[f_\ell' > 2 \ell \binom{r-1}{\ell}t^{r-\ell-1}q^{\ell-2}f_V\]
        \item \[f_{\ell}^{\prime}>7(\ell-1)\binom{r-1}{\ell-1} t^{r-\ell} q^{\ell-2}f_2 \]
        \item \[f_{\ell}^{\prime}>6(\ell-1)(r-1) t^{r-2} f_{\ell}\]
        \item \[f_{\ell}^{\prime}>3(\ell-1)(r-1)\binom{r-1}{\ell-1} t^{2 r-\ell-2} q^{\ell-2} f_V\]
        \item  \[\phi^\delta f_V = o(q)\]
        \item  \[q - \phi^\delta f_V \ge \phi^\lambda \]
        \item  \[q', q'' = O\rbrac{1}\]
        \item  \[s_\ell, (s_\ell^\pm)', (s_\ell^\pm)'' = O\rbrac{D^{\frac{\ell-1}{r-1}}}\]
        \item \[f_\ell, f_\ell', f_\ell'' = O(\phi^{-\lambda})\]
        \item \[f_V, f_V', f_V'' = O(\phi^{-\lambda})\]
    \end{enumerate}
\end{lemma}
\begin{proof}
First we note that we can choose $\alpha, \beta$ to satisfy (1)--(6). Indeed, these bounds are precisely the same as the variation equations in the first author and Bohman \cite{BennettBohmanNoteOnRandom} (see equations (13)--(18) in that paper). Finally we will choose $\zeta>0$ to be sufficiently small. Since we choose $\zeta$ last, it can depend on $r, \alpha, \beta, \delta, \lambda$. Now we justify (7)--(12). Recall from \eqref{eq:q} that $q \coloneqq e^{-t^{r-1}}$, thus
    \begin{enumerate}
    \setcounter{enumi}{6}
        \item  
        \[\frac{\phi^\delta f_V}{q} = \phi^{\delta}(1+t^2)\exp(\alpha t + (\beta-1)t^{r-1}) = O\rbrac{\phi^\delta \log(1/\phi) \phi^{-(\alpha + \beta -1)\zeta^{r-1}}} =o(1),\]
        where the final bound follows since we choose some small $\zeta>0$ after having chosen the other constants. For the rest of the proof we will use similar reasoning without comment. 
        \item This follows from (7) and the fact that $q(t) \ge q(t_{\max}) = \phi^{\zeta^{r-1}}.$
        \item $q'$ and $q''$ both have the following form: a polynomial in $t$ times $q$. For any polynomial $P(t)$, the continuous and differentiable function $P(t)q(t) \rightarrow 0$ as $t \rightarrow \infty$, which suffices to conclude that $P(t)q(t)=O(1)$ for $t\in [0, \infty)$.
        \item $s_\ell(t)$ can be written as $D^\frac{\ell-1}{r-1}P(t)[q(t)]^{\ell-1}$ for a polynomial $P$, and so by the same reasoning as the previous bound we get $s_{\ell}(t) = O\rbrac{D^{\frac{\ell-1}{r-1}}}$. Likewise by \eqref{eq:sellpmdef} we have 
        \[(s_\ell^+)' = \frac{D^{-\frac{1}{r-1}} \ell s_{\ell+1}}{q} = D^\frac{\ell-1}{r-1}P(t)[q(t)]^{\ell-1}\]
        for some polynomial $P(t)$ and so we get the bound on $(s_\ell^+)'$. Similarly $(s_\ell^-)', (s_\ell^+)'', (s_\ell^-)''$ all have the same form and so their bounds follow as well. 
        \item $f_\ell, f_\ell', f_\ell''$ all have the following form: $P(t) \exp(\alpha t + (\beta-\ell)t^{r-1})$, where $P(t)$ is some polynomial in $t$. Thus, $P(t)$ will be $polylog(\phi)$ for all $t \le t_{max}$, and we have 
        \[
        P(t) \exp(\alpha t + (\beta-\ell)t^{r-1}) = O\rbrac{polylog(\phi) \phi^{ -(\alpha + \beta)\zeta^{r-1})}} = O(\phi^{-\lambda})
        \]
        \item This follows similarly to the last bound. 
    \end{enumerate}
\end{proof}

\section{Deferred proofs of crude bounds}\label{sec:deferred}

First, we show that Condition \eqref{eq:DynamicSetDegBound} holds until step $i_{\max}$ a.a.s.~for $a=1$. Note that the case $b=r$ follows from Condition \ref{cond:almostregular} and the fact that we never create new edges of size $r$. So we assume $b \le r-1$.

\begin{lemma}\label{lem:vertexdegreeBound}  Let $2 \leq b \le r-1$, and recall that $D_{1 \uparrow b}  \coloneqq  D^{\frac{b-1}{r-1}} \phi^{-(2r-2b) \lambda}$. Then, 
\be\label{eq:vertexupperdegreeprobabilitybound}
\mathbb{P}\left(\exists i \leq \tau \text { and } v \in V(i) \text { such that } d_{b}(v, i) \geq D_{1 \uparrow b}\right) =o(1) \eqpd
\ee
\end{lemma}

\begin{proof}

 We define variables $Y_{b}(v, i)$ as follows.

$$
Y_{b}(v, i)\coloneqq  \begin{cases}d_{b}(v, i)-    \frac{r}{N}  D^{\frac{b}{r-1}} \phi^{-(2 r-2 b -1) \lambda }  \cdot i & \text { if } \mathcal{E}_{i-1} \text { holds } \\ Y_{b}(v, i-1) & \text { otherwise }\end{cases}
$$

To show that $Y_{b}(v, i)$ is a supermartingale, it suffices to observe  
\[
\mathbb{E}\left[\Delta d_{b}(v, i) \right] \le \frac{b d_{b+1}(v, i)}{|V(i)|}\le \frac{r}{N}  D^{\frac{b}{r-1}} \phi^{-(2 r-2 b -1) \lambda }.
\]

 For our application of Freedman's theorem \eqref{lem:FreedmanLemma}, we can set $C  =D_{2 \uparrow b+1} =D^{\frac{b-1}{r-1}}\phi^{1-(2 r-2 b-2) \lambda}$. Now

$$
\begin{aligned}
\mathbb{V} a r\left[\Delta Y_{b}(v, i)\right]=\mathbb{V} a r\left[\Delta d_{b}(v, i)\right] & \leq \mathbb{E}\left[(\Delta d_{b}(v, i))^2\right] \\
& \leq C\cdot \mathbb{E}\left[\left|\Delta d_{b}(v, i)\right| \right] \\
& \leq \frac{r}{N} D^{\frac{2 b-1}{r-1}} \phi^{1-(4 r-4 b-3) \lambda}
\end{aligned}
$$

Using \eqref{eq:i_max} we have $W({i_{\max}}) \leq i_{\max} \cdot \frac{r}{N} D^{\frac{2 b-1}{r-1}} \phi^{1-(4 r-4 b-3) \lambda} \le D^{\frac{2 b -2}{r-1}} \phi^{2-(4 r-4 b-3) \lambda} \log(1/\phi)$, so we can take $w$ to be the latter expression. Setting $d= \frac12 D^{\frac{b-1}{r-1}}\phi^{-(2r-2b)\lambda} $ and noting that $Y_{b}(v, 0)\le \Delta_1(\mathscr{H}^{(b)}) \le  \phi D^{\frac{b-1}{r-1}}$ by Condition \eqref{cond:max1degree},  Freedman's theorem gives us that 

\begin{align}
\begin{split}
P\left[\exists i \le i_{\max} : Y_{A \uparrow b}(i) - \phi D^{\frac{b-1}{r-1}} \geq d \right]& \le \exp \left(- \frac{d^2}{2(w+C d)}\right)\\
&\le \exp \left(-\Omega \left(\frac{D^{\frac{2(b-1)}{r-1}}\phi^{-(4r-4b)\lambda}}{D^{\frac{2 b -2}{r-1}} \phi^{2-(4 r-4 b-3) \lambda} \log(1/\phi)+D^{\frac{2b-2}{r-1}}\phi^{1-(4r-4b-2)\lambda}}\right)\right) \\
& < \exp \rbrac{-\phi^{-\lambda/2}} = o(N^{-1}).
\end{split}
\end{align}
If the unlikely event above does not happen, i.e. if $Y_{b}(v,i)-\phi D^{\frac{b-1}{r-1}} <d$ for all $i \le i_{\max}$, then 
\[
d_{b}(v, i) \le \phi D^{\frac{b-1}{r-1}}+ d + \frac{r}{N}  D^{\frac{b}{r-1}} \phi^{-(2 r-2 b -1) \lambda } \cdot i_{\max} < D_{1 \uparrow b}.
\]
 Now \eqref{eq:vertexupperdegreeprobabilitybound} follows from the union bound over at most $N$ choices for $v$.

\end{proof}

We next show that condition \eqref{eq:DynamicCoDegBound} also holds to step $i_{\max}$ a.a.s. 
\begin{lemma} \label{lem:coDegreeBound}
Let $2 \leq a, a^{\prime} \leq r$ and $1 \leq k<a, a^{\prime}$ be fixed. Then, recalling that $C_{a, a^{\prime} \rightarrow k}  \coloneqq 2^r D^{\frac{a+a^{\prime}-k-2}{r-1}}\phi^{1-(4 r-2 k-2) \lambda}$, we have that 
\be \label{eq:codegprobabilitybound}
\mathbb{P}\left(\exists i \leq \tau \right. \mbox{ and a non-bad pair } v, v^{\prime} \mbox{ such that } \left.c_{a, a^{\prime} \rightarrow k}\left(v, v^{\prime}, i\right) \geq C_{a, a^{\prime} \rightarrow k}\right) =o(1)\eqpd
\ee
 \end{lemma}

\begin{proof}

We first note that the definition of the good event implies Lemma \ref{lem:coDegreeBound} except in the case $a=a^{\prime}=k+1$. To see this, suppose $a^{\prime}>k+1$ and let $v, v^{\prime}$ be any two vertices. Then we have
$$
c_{a, a^{\prime} \rightarrow k}\left(v, v^{\prime}\right) \leq D_{1 \uparrow a} \cdot\binom{a-1}{k} \cdot D_{k+1 \uparrow a^{\prime}} \leq  2^r \cdot D^{\frac{a+a^{\prime}-k-2}{r-1}}\phi^{1-\left(4r-2a - 2a^{\prime}\right) \lambda},
$$
which gives the desired bound.
So we restrict our attention to the case $a=a^{\prime}=k+1$.  The  case of $k = r-1$ holds by the assumption since $v, v'$ is not a bad pair, and thus their $(r-1)$-codegree is less than $\phi D$. Now we assume $k \le r-2$.  

 We define variables $Y_{a, a^{\prime} \rightarrow k}\left(v, v^{\prime}, i\right)$ as follows.

$$
Y_{a, a^{\prime} \rightarrow k}\left(v, v^{\prime}, i\right)\coloneqq  \begin{cases}c_{a, a^{\prime} \rightarrow k}\left(v, v^{\prime}, i\right)-\frac{3 \cdot 2^r}{N} D^{\frac{k+1}{r-1}}\phi^{1-(2 r-2 k-1) \lambda}\cdot i & \text { if } \mathcal{E}_{i-1} \text { holds } \\ Y_{a, a^{\prime} \rightarrow k}(v, v^{\prime},i-1) & \text { otherwise }\end{cases}
$$

We justify that $Y_{a, a^{\prime} \rightarrow k}\left(v, v^{\prime}, i\right)$ is a supermartingale. Note that $c_{k+1, k+1 \rightarrow k}\left(v, v^{\prime}, i\right)$ can increase in size only when the algorithm chooses a vertex contained in the intersection of a pair of edges from $c_{k+2, k+2 \rightarrow k+1}\left(v, v^{\prime}, i\right)$, or when the algorithm chooses the vertex not contained in the intersection of a pair of edges counted by $c_{k+2, k+1 \rightarrow k}\left(v, v^{\prime}, i\right)$ or $c_{k+1, k+2 \rightarrow k}\left(v, v^{\prime}, i\right)$. Thus 

\[
\begin{split}
\mathbb{E}\left[\Delta c_{a, a^{\prime} \rightarrow k}\left(v, v^{\prime}, i\right) \right] \le \frac{C_{k+2, k+2 \rightarrow k+1}+C_{k+2, k+1 \rightarrow k}+C_{k+1, k+2 \rightarrow k}}{|V(i)|}&\leq \frac{2D^{\frac{k}{r-1}}\phi^{1-(2r-2k-4)\lambda}+D^{\frac{k-1}{r-1}}\phi^{1-(2r-2k-2)\lambda}}{N\phi^{\lambda}}\\
&\leq \frac{D^{\frac{k}{r-1}}\phi^{1-(2r-2k-2)\lambda}\left(2\phi^{2\lambda}+1\right)}{N\phi^{\lambda}}\\
&\le \frac{3 \cdot 2^r}{N} D^{\frac{k+1}{r-1}}\phi^{1-(2 r-2 k-1) \lambda}\eqpd
\end{split}
\]
Where, to get from the second to the third line we use that $\phi^{2\lambda}=o(1)$, and so $\left(2\phi^{2\lambda}+1\right) \leq 3$.
 For our application of Freedman's theorem \eqref{lem:FreedmanLemma}, we can set $C  =2 D_{2 \uparrow k+2}+D_{2 \uparrow k+1} \leq 3 D^{\frac{k}{r-1}} \phi^{1-( 2r- 2k-4) \lambda}$. Now

$$
\begin{aligned}
\mathbb{V} a r\left[\Delta Y_{a, a^{\prime} \rightarrow k}\left(v, v^{\prime}, i\right)\right]=\mathbb{V} a r\left[\Delta c_{a, a^{\prime} \rightarrow k}\left(v, v^{\prime}, i\right)\right] & \leq \mathbb{E}\left[(\Delta c_{a, a^{\prime} \rightarrow k}\left(v, v^{\prime}, i\right))^2\right] \\
& \leq C\cdot \mathbb{E}\left[\left|\Delta c_{a, a^{\prime} \rightarrow k}\left(v, v^{\prime}, i\right)\right| \right] \\
& \leq \frac{9 \cdot 2^r}{N}  D^{\frac{2 k+1}{r-1}}\phi^{2 -(4 r-4 k-5) \lambda}
\end{aligned}
$$

Thus we have $W({i_{\max}}) \leq i_{\max} \cdot \frac{9 \cdot 2^r}{N}  D^{\frac{2 k+1}{r-1}}\phi^{2 -(4 r-4 k-5) \lambda} =O\left( D^{\frac{2 k}{r-1}}\phi^{2-(4 r-4 k-5) \lambda} \log (1/\phi)\right)$ so we can take $w$ to be the latter expression. We will set $d=D^{\frac{k}{r-1}}\phi^{1-(4 r-2 k-2) \lambda} $. Observe that $Y_{a, a^{\prime} \rightarrow k}(v, v^{\prime}, 0) \le \phi D^{\frac{k}{r-1}}$ by Condition \eqref{cond:max1degree}. Freedman's theorem gives us that 

\begin{align}
\begin{split}
P\left[\exists i \le i_{\max} :  Y_{a, a^{\prime} \rightarrow k}\left(v, v^{\prime}, i\right) -\phi D^{\frac{k-1}{r-1}} \geq d  \right] &\le \exp \left(-\frac{d^2}{2(w+C d)}\right)\\
&=\exp \left(-\Omega \left(\frac{D^{\frac{2k}{r-1}}\phi^{2-(8r-4k-4)\lambda}}{D^{\frac{2 k}{r-1}}\phi^{2-(4 r-4 k-5) \lambda} \log (1/\phi) + D^{\frac{2k}{r-1}}\phi^{2-(6 r-4 k-6) \lambda} }\right)\right) \\
& \le \exp \left(- \phi^{-(2r+2)\lambda} \right) = o(N^{-2}).
\end{split}
\end{align}
If the unlikely event above does not happen, i.e. if $Y_{a, a^{\prime} \rightarrow k}\left(v, v^{\prime}, i\right)-\phi D^{\frac{k}{r-1}} <d$ for all $i \le i_{\max}$, then 
\[
c_{a, a^{\prime} \rightarrow k}\left(v, v^{\prime}, i\right) \le \phi D^{\frac{k}{r-1}}+d + \frac{3 \cdot 2^r}{N} D^{\frac{k+1}{r-1}}\phi^{1-(2 r-2 k-1) \lambda} \cdot i_{\max} < C_{k+1, k+1 \rightarrow k}.
\]
 Now \eqref{eq:codegprobabilitybound} follows from the union bound over at most $N^2$ choices for $v, v'$. 
 
\end{proof}

We now bound the probability that $\mathcal{E}_{i_{\text {max }}}$ fails due to condition \eqref{eq:Dynamicddubprime}.

\begin{lemma} Recall that $D''_{\ell}  \coloneqq  D^{\frac{\ell-1}{r-1}}\phi^{1-(2r-2\ell) \lambda}$. Then,\label{lem:degreedubprimeBound}
\be\label{eq:dubprimeprobabilitybound}
\mathbb{P}\left(\exists i \leq \tau \text { and } v \in V(i) \text { such that } d''_{\ell}(v,i) \geq D''_\ell \right) =o(1) \eqpd
\ee
\end{lemma}

\begin{proof}

 We define variables $Y_\ell(v, i)$ as follows.

$$
Y_\ell(v, i)\coloneqq  \begin{cases}d''_{\ell}(v, i)-  \frac{r}{N}  D^{\frac{\ell}{r-1}} \phi^{1-(2 r-2 \ell -1) \lambda } \cdot i & \text { if } \mathcal{E}_{i-1} \text { holds } \\ Y_\ell(v,i-1) & \text { otherwise }\end{cases}
$$

To justify that $Y_\ell(v, i)$ is a supermartingale, we observe
\[
\mathbb{E}\left[\Delta d''_\ell(v, i) \right] = \frac{(\ell-2)d''_{\ell+1}}{|V(i)|}\le \frac{r}{N}  D^{\frac{\ell}{r-1}} \phi^{1-(2 r-2 \ell -1) \lambda }.
\]

 For our application of Freedman's theorem \eqref{lem:FreedmanLemma}, we can set $C  =D_{2 \uparrow \ell+1} =D^{\frac{\ell-1}{r-1}}\phi^{1-(2 r-2 \ell-2) \lambda}$. Now

$$
\begin{aligned}
\mathbb{V} a r\left[\Delta Y_\ell(v, i)\right]=\mathbb{V} a r\left[\Delta d''_\ell(v, i)\right] & \leq \mathbb{E}\left[(\Delta d''_\ell(v, i))^2\right] \\
& \leq C\cdot \mathbb{E}\left[\left|\Delta d''_\ell(v, i)\right| \right] \\
& \leq \frac{r}{N} D^{\frac{2 \ell -1}{r-1}} \phi^{2-(4 r-4 \ell-3) \lambda}
\end{aligned}
$$

Thus we have $W({i_{\max}}) \leq i_{\max} \cdot \frac{r}{N} D^{\frac{2\ell -1}{r-1}} \phi^{2-(4 r-4 \ell-3) \lambda} < D^{\frac{2 \ell -2 }{r-1}} \phi^{2-(4 r-4 \ell-3) \lambda} \log(1/\phi)$ so we can take $w$ to be the latter expression. We set $d=\frac 12 D''_\ell = \frac12 D^{\frac{\ell-1}{r-1}}\phi^{1-(2r-2\ell)} $. We claim that $Y_\ell(v, 0) \le \phi D^{\frac{\ell-1}{r-1}}$. Indeed, for $\ell=r$ it follows from Condition \eqref{cond:ddubprime}, and for $\ell<r$ it follows from Condition \eqref{cond:max1degree}. Freedman's theorem gives us that 

\begin{align*}
\begin{split}
&P\left[\exists i \le i_{\max} : Y_\ell(v, i)-\phi D^{\frac{\ell-1}{r-1}} \geq d \right] \le \exp \left(-\frac{d^2}{2(w+C d)}\right)\\
&\qquad \le \exp \left(-\Omega \left( \frac{ D^{\frac{2\ell-2}{r-1}}\phi^{2-(4r-4\ell)}}{D^{\frac{2 \ell -2 }{r-1}} \phi^{2-(4 r-4 \ell-3) \lambda} \log(1/\phi) + D^{\frac{\ell-1}{r-1}}\phi^{1-(2 r-2 \ell-2) \lambda}\cdot D^{\frac{\ell-1}{r-1}}\phi^{1-(2r-2\ell)}} \right)\right) \\
&\qquad < \exp \rbrac{-\phi^{-\lambda}} = o(N^{-1}).
\end{split}
\end{align*}
If the unlikely event above does not happen, i.e. if $Y_\ell(v, i) -\phi D^{\frac{\ell-1}{r-1}}<d$ for all $i \le i_{\max}$, then 
\[
d''_\ell(v, i) \le \phi D^{\frac{\ell-1}{r-1}}+d + \frac{r}{N}  D^{\frac{\ell}{r-1}} \phi^{1-(2 r-2 \ell -1) \lambda }  \cdot i_{\max} < D''_\ell.
\]
 Now \eqref{eq:dubprimeprobabilitybound} follows from the union bound over at most $N$ choices for $v$.

\end{proof} 

\end{document}